\documentclass{amsart}[12]

\usepackage{amssymb,amsfonts,amsmath,amsopn,amstext,amscd,color,latexsym,mathrsfs,skak,stackrel,url,verbatim,amsthm}
\usepackage{mathabx}
\usepackage[utf8]{inputenc}
\usepackage[all, cmtip]{xy}
\usepackage{mathtools}
\usepackage{enumitem}
\usepackage[active]{srcltx}

\usepackage[pdftex, colorlinks=true, linktoc=all]{hyperref}

\usepackage{cleveref}

\theoremstyle{plain}
\newtheorem{thm}{\bf Theorem}[section]
\newtheorem{lem}[thm]{\bf Lemma}
\newtheorem{cor}[thm]{\bf Corollary}
\newtheorem{prop}[thm]{\bf Proposition}

\theoremstyle{remark}

\newtheorem{example}[thm]{\bf Example}

\newtheorem{rem}[thm]{\bf Remark}

\theoremstyle{definition}

\numberwithin{equation}{thm}

\newcommand{\bbA}{{\mathbb A}}

\newcommand{\bbC}{{\mathbb C}}

\newcommand{\bbF}{{\mathbb F}}
\newcommand{\bbG}{{\mathbb G}}

\newcommand{\bbN}{{\mathbb N}}

\newcommand{\bbP}{{\mathbb P}}
\newcommand{\bbQ}{{\mathbb Q}}

\newcommand{\bbZ}{{\mathbb Z}}

\newcommand{\cA}{{\mathcal A}}

\newcommand{\cC}{{\mathcal C}}
\newcommand{\cD}{{\mathcal D}}

\newcommand{\cF}{{\mathcal F}}
\newcommand{\cG}{{\mathcal G}}
\newcommand{\cI}{{\mathcal I}}
\newcommand{\ccH}{{\mathcal H}}
\newcommand{\cL}{{\mathcal L}}

\newcommand{\cK}{{\mathcal K}}
\newcommand{\cM}{{\mathcal M}}

\newcommand{\cO}{{\mathcal O}}
\newcommand{\cP}{{\mathcal P}}

\newcommand{\cS}{{\mathcal S}}

\newcommand{\cX}{{\mathcal X}}
\newcommand{\cY}{{\mathcal Y}}

\newcommand{\rD}{{\rm D}}

\newcommand{\rK}{{\rm K}}

\newcommand{\sA}{{\mathscr A}}

\newcommand{\sC}{{\mathscr C}}
\newcommand{\sD}{{\mathscr D}}

\newcommand{\sM}{{\mathscr M}}

\newcommand{\sU}{{\mathscr U}}

\newcommand{\sX}{{\mathscr X}}
\newcommand{\sY}{{\mathscr Y}}
\newcommand{\sZ}{{\mathscr Z}}

\newcommand{\Alev}{{}_N\sA_g}
\newcommand{\Alevn}[1]{{}_N\sA_{#1}}
\newcommand{\toA} {{\overline{\Alev}}}
\newcommand{\cotan}{T^*}

\DeclareMathOperator{\Spec}{Spec}

\DeclareMathOperator{\Pic}{Pic}

\DeclareMathOperator{\im}{Image}

\DeclareMathOperator{\Perv}{Perv}

\newcommand{\isom}{\cong}

\begin{document}
\title[signed Euler characteristics of moduli]{Nonnegativity of signed Euler characteristics of moduli of curves and abelian varieties.}
\author{Donu Arapura}
\address{Department of Mathematics, Purdue University,
150 N. University Street, West Lafayette, IN 47907, U.S.A.}
\email{arapura@purdue.edu}
\author{Deepam Patel}
\address{Department of Mathematics, Purdue University,
150 N. University Street, West Lafayette, IN 47907, U.S.A.}
\email{patel471@purdue.edu}
\thanks{First author supported by a grant from the Simons Foundation}

\begin{abstract}
Formulas of Harer-Zagier and Harder imply  that the orbifold Euler characteristic of the moduli stacks of smooth curves or principally polarized abelian varieties have a sign given by $(-1)^{\dim \sM_g }$
or   $(-1)^{\dim \sA_g }$. This is generalized as follows:
Given a perverse sheaf on a product of moduli stacks of either type, it is proved that the  Euler characteristic is nonnegative when the base field has characteristic zero. This is shown to be false in positive characteristic.
\end{abstract}
\maketitle

\tableofcontents

\section{Introduction}
Let $\sM_{g,n}$ denote the moduli stack of smooth projective curves of genus $g$ and $n$-marked points, let $\sA_g$ denote the moduli stack of principally polarized $g$-dimensional abelian varieties, and let 
$\sA_{g,n}$ the $n$-fold fibre product of the universal family over $\sA_g$. We work over an algebraically closed base field $k$. In this setting, when $k = \bbC$, we have the following results due to Harer-Zagier \cite{hz} and Harder \cite{harder}:

\begin{enumerate}
\item The orbifold characteristic or  equivalently the Euler-Satake characteristic 
$$\chi^{orb}(\sM_{g,n})= (-1)^n\frac{(2g-1)B_{2g}}{(2g)!}(2g+n-3)!\quad \text{when } g\ge 2$$
\item If  $n \geq 3$, then  $\chi^{orb}(\sM_{0,n})=(-1)^{n-3}(n-3)!$.
\item If  $n \geq 1$, then  $\chi^{orb}(\sM_{1,n})=(-1)^n\frac{(n-1)!}{12}$
\item The orbifold Euler characteristic 
$$\chi^{orb}(\sA_{g})= \zeta(-1)\zeta(-3)\cdots\zeta(1-2g).$$
\end{enumerate}

\begin{cor}
 The sign of $\chi^{orb}(\sM_{g,n}) $ is given by $(-1)^{\dim \sM_{g,n}}$ and the sign of $\chi^{orb}(\sA_g)$ is given by $(-1)^{\dim(\sA_g)}$
\end{cor}

\begin{proof}
 This follows from the standard formulas $\dim(\sM_{g,n}) = 3g-3+n$, $\dim (\sA_{g})  = \frac{g(g+1)}{2}$, $\zeta(1-2m) = -\frac{B_{2m}}{2m}$, and
 $$\operatorname{sign} B_{2m} =
\begin{cases}
 +1 & \text{if $m$ odd}\\
 - 1 &\text{if $m$ even}
\end{cases}
$$
(We take  $\operatorname{sign} x$ to be $1$ if $x \geq 0$ and $-1$ otherwise.)
\end{proof}

In this article, we generalize this corollary by showing that the sign of the Euler characteristic of any smooth closed irreducible substack  $\sX$ of one the previous stacks satisfies the same condition.  
In fact, we can strengthen this as follows.
Given a Deligne-Mumford stack defined over $k$,
let $\Perv(\sX)$ be the category of perverse sheaves. In the following, this will mean one of two things. If $k = \bbC$, one can consider the underlying analytic Deligne-Mumford stack and perverse sheaves with $\bbQ$ coefficients (in the analytic topology). For an arbitrary field $k$ with $char(k) =p$, one can consider the category of $\ell$-adic perverse sheaves (for some $\ell \neq p$). In this setting  and when $char(k)=0$, one can define the Euler-Satake characteristic of a perverse sheaf (as a rational number). The precise definition of $\Perv(\sX)$ and  the notion of Euler characteristic for an object
of this category is given in  section \ref{sec:background}.
Let us start by stating one of our main results.

\begin{thm}\label{thm:mainthm}
Let $\sX$ be one of $\sM_{g,n}$ (with $2g-2+n\geq 3$) or $\sA_{g,n}$, and suppose $char(k) = 0$. Then for $\cK \in \Perv(\sX)$
$$ \chi^{orb}(\sX,\cK) \geq 0.$$
\end{thm}

Another motivation for this theorem comes from the work of Liu, Maxim and Wang \cite{LMW}, the author and Wang \cite{aw}, and Deng and Wang \cite{dw}. Similar inequalities were obtained
in \cite{aw} and \cite{dw} (and conjectured in \cite{LMW})  for perverse sheaves on a projective manifold that admits a  nonpositively curved K\"ahler metric or a {\em large} local system. A local system is large  if it  restricts nontrivially to all curves mapping nontrivially to the variety.
These conditions certainly hold for the orbifold $\sA_g$ since it is a locally symmetric space, and for $\sM_g$ since it essentially embeds into it.
So theorem \ref{thm:mainthm} should be viewed as an extension of the aforementioned results to these sorts of  noncompact spaces.

\begin{cor}\label{cor:maincor}
   If $\sX\subset \sA_{g,n}$ or $\sX\subset \sM_{g,n}$ is a smooth closed irreducible substack, then
   $$(-1)^{\dim X}\chi^{orb}(\sX)\ge 0.$$
   More generally, the analogous assertion holds for any local complete intersection.
\end{cor}

\begin{proof}
    Take $\cK= \iota_*\bbQ_{\sX}[\dim \sX]$ where $\iota$ is the inclusion of $\sX$. This  is a perverse sheaf by \cite[chap III, lemma 6.5]{kw}.
\end{proof}

\begin{cor}
 If  $\cK$ is a perverse sheaf on a Shimura variety $\sX$ of Hodge type and finite level, then $\chi^{orb}(\sX, \cK)\ge 0$.
\end{cor}

\begin{proof}
 Essentially, by definition, $\sX$ will embed as a closed smooth substack of $\sA_{g}$ for some $g$.
\end{proof}

In particular, this theorem gives a simple proof that recovers the sign of the Euler characteristic of $\sM_{g,n}$ and $\sA_{g,n}$. 
We expect that the last corollary should hold for arbitrary Shimura varieties. For compact Shimura varieties, this follows from \cite{aw}.

Before we explain the basic idea of the proofs, we should  note that the above theorem is in fact a combination of special cases of two  separate theorems. The first result involves abelian varieties.

\begin{thm}\label{thm:intro2}
Let $\cK$ be a perverse sheaf on $\sA_{g_1,n_1}\times  \ldots \times \sA_{g_r,n_r}$, and suppose $char(k)=0$. Then $$\chi^{orb}(X,\cK) \geq 0.$$
\end{thm}

Using moduli spaces $\prescript{}{N}{\sA_{g,n}}$  with level structures $N\ge 3$, we reduce to an analogous statement about varieties rather than stacks. To check
this version of the theorem, we verify that the log cotangent bundle of a product of toroidal compactifications
$$\overline{X}=\overline{{}_{N_1}\sA_{g_1, n_1}}\times  \ldots \times \overline{{}_{N_r}\sA_{g_r, n_r}}$$
is nef. This is reduced 
to a formula of Faltings and Chai, which expresses the
log cotangent bundle of each factor as the second symmetric power of a Hodge
bundle, and the Fujita-Kawamata semipositivity theorem for the latter.
The above theorem now follows from the nefness result together with the log Dubson-Kashiwara theorem and an inequality of Fulton-Lazarsfeld. This last step of the proof is similar to and inspired by the arguments used in the papers \cite{aw, dw, LMW} mentioned above. We also note that the idea of using the Dubson-Kashiwara formula to show positivity of the Euler characteristic of perverse sheaves first appears (in the context of semi-abelian varieties) in the work of Franecki-Kapranov (\cite{fk}). 

Our second  result concerns moduli of  curves.  It is convenient to work with the moduli stack 
 $\sM_{g,n}^c\subset \overline{\sM_{g,n}}$ of genus $g$ stable curves of compact type with $n$ marked points, and its substack $\cM_{g,n}^{sc}$.\footnote{By definition, this is the open substack of compact type curves such that any genus zero vertex is only connected to higher genus vertices. See Section \ref{sec:torelli}.}
One has
 $\sM_{g,n}\subset \sM_{g,n}^{sc} \subset \sM_{g,n}^{c}$.
 Our second main theorem is the following:

\begin{thm}\label{thm:intro3}
 Let $\sX= \sM_{g_1, n_1}^{sc}\times\ldots \times \sM_{g_r, n_r}^{sc}$ where $2g_i-2+n \geq 3$, and suppose $char(k)=0$.
   If   $\cK$ is a perverse sheaf on $\sX$, then
    $$\chi^{orb}(\sX,\cK)\ge 0$$
    
    \end{thm}

We will define a proper morphism
$$\tau:{}_N\sM_{g,n}^c\to {}_{N}\sA_{g, n}$$
which agrees with the Torelli map on ${}_N\sM_{g,n}$. An application of Theorem \ref{thm:intro2} shows that the conclusion of the previous theorem does hold for $Y=\tau({}_N\sM_{g,n}^{sc})$.
On the other hand, since $\tau$ restricted to ${}_N\sM_{g,n}$ is not an affine open embedding, Theorem \ref{thm:intro3} does not immediately follow. Our strategy is to induct on the boundary strata for which the stack $\cM_{g,n}^{sc}$ is a useful technical tool. The following technical result, 
proved in section \ref{sec:key} using
Beilinson's gluing construction (for perverse sheaves) and a deformation to the normal cone technique, allows us to deduce Theorem \ref{thm:intro3} via induction on boundary strata and Theorem \ref{thm:intro2}.

\begin{thm}\label{thm:intro4}
Suppose we are given a diagram of quasiprojective varieties
 $$
\xymatrix{
  & X\ar^{f}[d] \\ 
 U\ar^{\tilde j}[ru]\ar^{j}[r] & Y
}
$$
where $j,\tilde j$ are open immersions, $f$ is birational,
and $Z= X-U$ is a Cartier divisor. 
Suppose that for all $K'\in \Perv(Y)$ and $K''\in \Perv(Z)$ we have $\chi(Y, K')\ge 0$ and $\chi(Z, K'')\ge 0$.
 Then for all $K\in \Perv(X)$, we have $\chi(X,K)\ge 0$.
\end{thm}

For all but the final section of the paper, the ground field $k$ will have characteristic $0$.
In the final section, we show that the analog of Theorem \ref{thm:mainthm} fails in positive characteristic. Let us now view $\prescript{}{N}\sA_{g} $
as a scheme over $\Spec \bbZ[1/N]$, and 
let $\prescript{}{N}\sA_{g} \otimes \bar{\bbF}_p$ denote the fibre over $\bar{\bbF}_p$. We first observe that the ($\ell$-adic) Euler characteristic is independent of $p \nmid N$ (see \ref{thm:eulercharsameinp}).

\begin{thm}
For all $p \nmid N$, $\chi(\prescript{}{N}\sA_{g} \otimes \bar{\bbF}_p) = \chi(\prescript{}{N}\sA_{g} \otimes \bbC).$
\end{thm}

On the other hand, the existence of supersingular strata allows one to construct counter-examples to our main theorem in characteristic  p. In particular, we show in Corollary \ref{cor:counterexamplecharp}  that:

\begin{thm}
There exists a sequence of smooth subvarieties $V_{g_i}\subset \prescript{}{N}\sA_{g_i}\otimes \bar{\bbF}_p$ such that
 $(-1)^{\dim V_{g_i}} \chi(V_{g_i})\to -\infty$.
\end{thm}

This is closely related to the fact, first observed by Moret-Bailly \cite{mb}, that the Fujita-Kawamata theorem fails in positive characteristic. See Remark \ref{rem:counterexamplesemipositive}.
 We note that we work with  moduli with level structure, because  it is not clear how to define the Euler-Satake characteristic of a Deligne-Mumford stack in characteristic  $p$. \\

Finally, we briefly outline the contents of the various sections. The second section recalls some background results which will be used in the following sections. In particular, we recall a positivity result of Fulton-Lazarsfeld, a log version of the Dubson-Kashiwara formula, and apply this to prove semi-positivity of certain Hodge bundles. We also give the definition of the Euler-Satake characteristic of constructible sheaves on Deligne-Mumford stacks (following \cite{Beh}), and recall some facts on the perverse sheaves on stacks. In particular, we observe that the constant sheaf  on a local complete intersection stack is a perverse sheaf. In section 3, we prove the signed Euler characteristic property for ${}_N\cA_{g}$ and their products. In section 4, we prove analogous assertions for the universal family over ${}_N\cA_g$, and apply this to prove Theorem \ref{thm:intro2}.  In section 5, we prove Theorem \ref{thm:intro4}. In section 6, we construct an extension of the Torelli map to $\sM^c_{g,n}$.
Finally, these results are applied in section 7 to prove Theorem \ref{thm:intro3} and finish the proof of  Theorem \ref{thm:mainthm}. In section 8, we discuss the aforementioned results in characteristic  p.\\

{\bf Acknowledgements:} 
The first author thanks Laurentiu Maxim and Botong Wang for introducing him to some of these ideas in connection with the Hopf-Singer conjecture.

\section{Background}\label{sec:background}

In this section, we recall some standard results which will be useful in the following. 

\subsection{Fulton-Lazarsfeld for nef bundles}
Let $X$ be a smooth projective variety, and $E$ be a vector bundle of rank $e$.
Since the zero section $\iota:X\hookrightarrow E$ is a regular embedding,
we have a Gysin map $\iota^!: CH_e(E)\to CH_0(X)$. Let
$$\langle C, X\rangle = \operatorname{deg}  \iota^! [C]$$
for an $e$-cycle $C$ on $E$. This should be viewed as the intersection number of $C$ with the zero section.
A conical subset  $C\subset E$ is a Zariski closed subset invariant under the natural $\bbG_m$-action. 
In this setting, one has the following positivity of intersection products result due to Fulton-Lazarsfeld.

\begin{thm}\label{thm:fl}
Suppose that $E$ is a nef vector bundle of rank $e$. Then for any conical subset $C \subset E$ of dimension $e$,
$$\langle C, X \rangle \geq 0$$
\end{thm}
\begin{proof}
 See \cite{lazarsfeld}, 8.2.6.
\end{proof}


\subsection{Log Dubson-Kasiwara}\label{subsect:dubkash}
Suppose that $k=\bbC$.
Let $X$ be a smooth  variety of dimension $d$, and $\rD^b_c(X)$ denote the bounded derived category of constructible sheaves (of $\bbC$-vector spaces) on $X$. Given $\cK \in \rD^{b}_c(X)$, there is a well-defined cycle $CC(\cK) \in Z^{d}(\cotan{X})$, where $Z^{d}(\cotan(X))$ is the group of cycles of codimension $d$. This construction satisfies the following well known properties which we collect below:
\begin{enumerate}
\item The components of $CC(\cK)$ are conic Lagrangian subsets. In particular, they are of the form $\overline{\cotan_{Z_{reg}} X} \subset \cotan{X}$ where $Z \subset X$ is a closed subvariety and $Z_{reg}$ denotes the regular locus.
\item If $\cP$ is a perverse sheaf on $X$, then $CC(\cP)$ is an effective cycle.
\item (Dubson-Kashiwara Formula) Suppose furthermore, that $X$ is projective. Then, one has
$$ \chi(X,\cK) = \langle CC(\cK),X\rangle.$$
\end{enumerate}

In (\cite{wz}), Wu and Zhou give an extension of the Dubson-Kashiwara formula to the setting of log cotangent bundles. Suppose now that $(X,D)$ is a pair where $X$ is a smooth projective variety and $D$ is a simple normal crossings divisor. In this case, one have the standard exact sequence:
$$0 \rightarrow \Omega^1_X \rightarrow \Omega^1_X(\log D) \rightarrow \cO_D \rightarrow 0.$$
In particular, this gives an injective map of the corresponding vector bundles
$$\cotan X \hookrightarrow \cotan X(\log D).$$ If $U := X \setminus D$, then we obtain an open immersion
$$ \cotan U \hookrightarrow \cotan X(\log D).$$
Given $\cK \in \rD^{b}_c(U)$, let $\overline{CC(\cK)}$ denote its closure in $\cotan X(\log D)$. This is a conical subset of dimension $d$. In \cite[theorem 1.6]{wz}, the authors show that
\begin{equation}\label{eq:dubkash}
     \chi(U,\cK) = \langle \overline{CC(\cK)},X\rangle.
\end{equation}

Putting these results together yields the following fact needed below.

\begin{lem}\label{lem:pos}
Let $\bar A$ be a smooth projective variety with an snc divisor $E$ such that $\Omega_{\bar A}^1(\log E)$ is nef. If $P$ is a perverse sheaf on $A= \bar A-E$, then $\chi(A,P)\ge 0$.
\end{lem}

\begin{proof}

Let $\overline{CC(\cP)} \subset \cotan_{\bar A}(\log E)$ be the characteristic cycle defined in section \ref{subsect:dubkash}. This is effective with components given by closures of conormal bundles $\cotan_V A$ for $V \subset A$. 
%
The log Dubson-Kashiwara formula of Wu and Zhou \eqref{eq:dubkash}
shows that
$$\chi(A,P) = \langle \overline{CC(P)}, A \rangle$$
Therefore it is enough to show that this intersection number  is nonnegative.
 Since the log cotangent bundle is nef.
Theorem \ref{thm:fl} 
yields the desired nonnegativity.

\end{proof}

We will  need the following as well.

\begin{lem}\label{lem:prodnef}
 Let  $X_1$ and $X_2$ be smooth projective varieties with snc divisors $E_i$, such that $\Omega_{X_i}^1(\log E_i)$ are nef.
 Then $\Omega_{X_1\times X_2}^1(\log p_1^*E_1+p_2^* E_2)$ is nef, where $p_i:X_1\times X_2\to X_i$ are the natural projections.
\end{lem}

\begin{proof}
 This follows from \cite[theorem 6.2.12]{lazarsfeld} and the fact that
  $$\Omega_{X_1\times X_2}^1(\log p_1^*E_1+p_2^* E_2) = p_1^* \Omega_{X_1}^1(\log E_1)\oplus
  p_2^* \Omega_{X_2}^1(\log E_2)$$
\end{proof}


\subsection{K\"ahler hyperbolicity}
Given a K\"ahler manifold $X$, the universal cover $\pi:\tilde X\to X$ carries an induced K\"ahler metric.
The manifold $X$ is K\"ahler hyperbolic if the pullback of the K\"ahler form $\pi^*\omega= d\alpha$, where $\alpha$ has bounded norm with respect to the induced metric. We have the following examples \cite[ex 0.3A]{gromov}:

\begin{enumerate}
    \item  A Hermitian locally symmetric space of noncompact type is K\"ahler hyperbolic.
    \item A complex submanifold of a K\"ahler hyperbolic manifold is K\"ahler hyperbolic.
\end{enumerate}

\subsection{Log structures}

It is convenient to formulate some results in the language of log structures.  We will
not define them here, but instead refer the reader to Kato's paper \cite{kato} for details.
Given a smooth $k$-variety $Y$, a divisor $D\subset Y$ with  reduced normal crossings induces a log structure  \cite[ex 1.5]{kato}, which we denote by $(Y,D)$. We refer
to  a log structure of this type as having normal crossing type or just nc. 
If $f:X\to Y$ is a  morphism of smooth varieties, with reduced  divisors with normal crossings $D\subset Y$ and $E\subset X$ such that $f^{-1}D\subseteq E$, then we get a morphism of nc log schemes $f:(X,E)\to (Y,D)$. We say $f$ is log smooth if it is a smooth morphism of log schemes \cite[\S 3.3]{kato}. 
This implies that
$$\Omega^1_{(X,E)/(Y,D)} := \Omega_X^1(\log E)/f^*\Omega_Y^1(\log D)$$
is locally free of rank equal to the relative dimension of $X/Y$.
We note that nc log schemes are log smooth over $\Spec(k)$ (where $\Spec(k)$ has the trivial log structure).

\subsection{Hodge bundles}

Suppose that $(Y,D)$ is a  projective nc log scheme over $k$ as above.
Let $\overline{f}:(X,E)\to (Y,D)$ be a projective log smooth morphism to $(Y,D)$, which restricts to an abelian scheme $f:A\to Y-D$ over $Y -D$.
Let  $\overline{F}^1= \overline{f}_*\Omega^1_{(X,E)/(Y,D)}$.

\begin{thm}[Fujita-Kawamata]\label{thm:semipos}
    The sheaf  $\overline{F}^1$ is a nef vector bundle.
\end{thm}

\begin{proof}
Without loss of generality, we can assume that $k=\bbC$.
Switching to the analytic topology, 
  $R^1f^{an}_*\bbC$ underlies
a polarized variation of Hodge structure $(V, F^\bullet,\ldots)$ of type $\{(1,0), (0,1)\}$,
with unipotent monodromy about components of $D$.
The  Hodge bundle
$F^1V= f^{an}_*\Omega^1_{A/Y-D}$ extends to   $(\overline{F}^1)^{an}$ over $Y$.
The theorem now follows  from the  much more general result 
due to Fujino-Fujisawa-Saito \cite[theorem 3]{ffs}.  
\end{proof}

We will  refer to $\overline{F}^1$ as the Hodge bundle associated to $f:A\to Y-D$.

\subsection{Euler-Satake characteristic of constructible sheaves on a stack.}\label{sec:euler}

In this subsection, we explain the construction of the Euler-Satake (or orbifold Euler)  characteristic  for constructible sheaves on Deligne-Mumford or DM stacks. We will usually reserve the symbols $\sA, \sM, \sX,\ldots$ for stacks. A $1$-morphism between stacks will often be referred to simply as a morphism or a map.\\

Let $\sX$ be a DM stack of finite type over a field $k$, and $\rD^b_c(\sX,\Lambda)$ denote the derived category of constructible sheaves of $\Lambda$-modules in one of the following settings:
\begin{enumerate}
\item If $k =\bbC$, we take constructible sheaves in the analytic topology, and let $\Lambda$ be a field.
\item If  $k$ is arbitrary, we consider constructible sheaves in the lisse-etale topos (c.f. \cite{laumon, lo}) with $\Lambda = \bbZ/\ell^n\bbZ$ or $\bbQ_{\ell}$.
\end{enumerate}

In any of the aforementioned settings and when $\Lambda$ is a field, one can define the Euler-Satake  characteristic $\chi^{orb}(\cF)$ of $\cF \in \rD^b_c(\sX,\Lambda)$. In the analytic setting, a version of this, with $\cF$ replaced by a constructible function, was  defined by Behrend (\cite{Beh});
$\chi^{orb}$ was considered by several authors for constant coefficients, starting with Satake in the 1950's. The Euler-Satake  characteristic is characterized as follows:

\begin{thm}\label{thm:EulerSatake}
There exists a rational number $\chi^{orb}(\sX,\cF) \in \bbQ$ associated to a 
  DM stack $\sX$ of finite type over a field $k$ and $\cF \in \rD^b_c(\sX)$
  uniquely determined by the following properties:
\begin{enumerate}
\item[(i)] If $\sX=X$ is a scheme, then $\chi^{orb}(\sX,\cF) = \chi(X,\cF)$ i.e. the orbifold Euler characteristic is the usual Euler characteristic. 
\item[(ii)] Let $\sU \subset \sX$ be open, and $\sZ$ the closed complement. Then $\chi^{orb}(\sX,\cF) = \chi^{orb}(\sU,\cF|_{\sU}) + \chi^{orb}(\sZ,\cF|_{\sZ})$.
\item[(iii)] If $f: \sX \rightarrow \sY$ is a finite etale map of degree $d$, then $\chi^{orb}(\sX,f^*\cF) = d\chi(\sY,\cF)$. 

\item[(iv)] Let $\cF$ be a constructible sheaf on $\sX$ and $\cG$ on $\sY$. Then $\chi^{orb}(\sX \times \sY, \cF \boxtimes \cG) = \chi^{orb}(\sX,\cF)\chi^{orb}(\sY,\cG).$
\item[(v)] Let $\cF\to \ccH\to \cG\xrightarrow{[1]}$ be a distinguished triangle in $\rD_c^b(\sX,\Lambda)$, then $\chi^{orb}(\sX,\ccH) = \chi^{orb}(\sX,\cF) + \chi^{orb}(\sX,\cG).$
\item[(vi)] One has $\chi^{orb}(\sX,\cF[d]) = (-1)^d\chi^{orb}(\sX,\cF)$
\end{enumerate}
\end{thm}

\begin{proof}[Sketch]
By \cite[cor 6.1.1]{laumon} and quasicompactness, we can find a finite open cover
$\{\sU_1,\ldots, \sU_N\}$, where each $\sU_i \cong [U_i/G_i] $ is the quotient of a scheme by an etale group scheme. 
Therefore, we can use induction on  $N$  to prove the existence and uniqueness of $\chi^{orb}(X, \cF)$.
When $N=1$, axiom (iii) forces $\chi^{orb}([U_1/G_1], \cF) =\frac{1}{|G_1|} \chi(U_1, \cF_1)$,  where $\cF_1$ is the pull back of $\cF$ to $U_1$. Standard properties of Euler characteristics, will show that this number will satisfy the remaining axioms.
If $N>1$, set $\sZ = \sX\setminus \sU_1$, then $\chi^{orb}(\sZ, \cF|_{\sZ})$ is defined by induction.
Axiom (ii) now forces
$$\chi^{orb}(\sX, \cF) = \chi^{orb}(\sU_1, \cF|_{\sU_1}) + \chi^{orb}(\sZ, \cF|_{\sZ})$$
Since both terms on the right satisfy the above axioms, the same is true for the left hand side.
\end{proof}

We write
$$\chi^{orb}(\sX) = \chi^{orb}(\sX, \Lambda_X)$$

\subsection{Perverse sheaves on stacks}

In (\cite{lo}), Laszlo-Olsson  define perverse sheaves in the lisse-\'etale
topology of an Artin stack of finite type (over a fixed field $k$ of $char(k) =0$). We refer to loc. cit. for the details of this construction. Below, we recall some basic facts which will be useful in the following. We denote by $\Perv(\sX)$ the corresponding category of perverse sheaves (with coefficients in $\bbQ$ or $\bbQ_\ell$).

\begin{enumerate}
\item Let $j: \sU \hookrightarrow \sX$ be an open immersion. Then $j^*$ is $t$-exact.
\item If $\sZ \hookrightarrow \sX$ is a closed smooth irreducible substack (or more generally a local complete intersection), then $i_*(\bbQ[d])$ (with $d = \dim_{\sZ}$) is a perverse sheaf. Both statements follow directly from the definitions given the analogous statements for schemes, and the fact that local complete intersections are stable under smooth base change.
\end{enumerate}

\section{The case of ${}_N \sA_{g}$}

Fix $N\ge 3$, and let ${}_N \sA_{g}$ denote the moduli space of principally polarized $g$ dimensional abelian varieties with level $N$ structure  over $k$. This is a fine moduli space, and it is nonsingular. Let $\toA$ denote a nonsingular toroidal compactification such that $E= \toA-\Alev$ is a divisor with normal crossings (\cite{fc}). In this section, we prove the following theorem:

\begin{thm}\label{thm:chi}
Let $\cP$ be a perverse sheaf on $X={}_{N_1}\sA_{g_1}\times  \ldots \times {}_{N_r}\sA_{g_r}$ where $N_i\ge 3$. Then $\chi(X,\cP) \geq 0$.

\end{thm}

\begin{proof}
This follows from  Lemmas \ref{lem:pos} and \ref{lem:prodnef}, and Lemma  \ref{lem:OmegaAgNef} below.
\end{proof}


\begin{cor}
    If $X\subseteq  {}_{N_1}\sA_{g_1}\times  \ldots \times {}_{N_r}\sA_{g_r}$ ($N_i \geq 3$) is a smooth or lci subvariety, then
    $$(-1)^{\dim X}\chi(X)\ge 0.$$
\end{cor}

\begin{lem}\label{lem:OmegaAgNef}
Let $A:=\Alev$ and $\bar A=\toA$. Then
$\Omega_{\bar A}^1(\log E)$ is nef.    
\end{lem}
\begin{proof}
 Let $f:U\to A$ denote the universal family of abelian varieties, 
and let $\overline{F}^1$ be the associated Hodge bundle over  $\bar A$.
By \cite[chap IV, theorem 7.7]{fc}, 
$$\Omega_{\bar A}^1(\log E) \cong S^2\overline{F}^1$$
Therefore the lemma follows from theorem \ref{thm:semipos} and \cite[thm 6.2.12]{lazarsfeld}.
\end{proof}

We conclude the section with some Hodge theoretic variants.

\begin{thm}\label{thm:chiStrict}
   If $X\subset {}_{N_1}\sA_{g_1}\times  \ldots \times {}_{N_r}\sA_{g_r}$ is a smooth projective variety,
   then for all $p$
   $$(-1)^{\dim X -p}\chi(\Omega_X^p)>0$$
\end{thm}

\begin{proof}
    By \cite[0.3A]{gromov} $X$ is K\"ahler hyperbolic, therefore the theorem follows from \cite[theorem 2.5]{gromov}.
\end{proof}

\begin{cor}
    If $X\subset {}_{N_1}\sA_{g_1}\times  \ldots \times {}_{N_r}\sA_{g_r}$ is a smooth projective variety, then
    $$(-1)^{\dim X}\chi(X)> \dim X$$
\end{cor}

\begin{proof}
    By the Hodge decomposition
    $$\chi(X)= \sum_{p=0}^{\dim X} (-1)^p\chi(\Omega_X^p)$$
\end{proof}

\begin{rem}
Since the boundary of the Satake compactification of ${}_N\sA_{g}$ has codimension $g$,
it follows that it has plenty of positive dimensional smooth projective subvarieties.
\end{rem}

\section{Euler characteristics of perverse sheaves on $\sA_{g,n}$}

 Let $f:(\overline{X}, E)\to (\overline{Y},D)$ be a projective log smooth map of nc log schemes.
 Set $\cX= (\overline{X},E)$, $\cY= (\overline{Y},D)$,
 $\Omega^1_\cX= \Omega_{\overline{X}}^1(\log E)$ etc. We have an exact sequence

\begin{equation}\label{eq:TXY}
     0  \to  f^*\Omega^1_\cY \to  \Omega^1_\cX\to  \Omega^1_{\cX/\cY} \to 0
\end{equation}

\begin{lem}\label{lem:logsmooth}
 In the above situation, suppose that
 $\Omega^1_\cY$ is nef and that $\Omega_{\cX/\cY}^1=f^*V$, where $V$ is
 nef. Then $\Omega_\cX^1$ is nef.

\end{lem}

\begin{proof}
This follows from the fact that the nef property is closed under pull backs and extensions
\cite[thm 6.2.12]{lazarsfeld}.

\end{proof}

Fix $N\ge 3$.
Since $\Alev$ is a fine moduli space, we have a universal abelian scheme ${}_N\sA_{g,1}\to \Alev$.
Let $\Alevn{g,n} = \Alevn{g,1}\times_{\Alev} \ldots \times_{\Alev}  \Alevn{g,1}$ denote the $n$-fold fibre product.
This a fine moduli space, whose $S$-valued points are $g$ dimensional abelian schemes of $S$ with a principal
polarization, level $N$ structure, and $n$ sections. We let $\sA_{g} = [\Alevn{g }/PSp_{2g}(\bbZ/N\bbZ)]$,
$\sA_{g,n} = [\Alevn{g,n }/PSp_{2}(\bbZ/N\bbZ)]$
denote the moduli  stacks of principally polarized $g$ dimensional abelian schemes (with $n$ sections).

\begin{thm}\label{thm:NAgn}
Let $X= \Alevn{g_1, n_1}\times \ldots  \times \Alevn{g_r, n_r}$.
Then there is an nc log scheme $(\overline{X}, E)$
such that $\Omega_{(\overline{X}, E)}^1$ is nef, $\overline{X}$ is proper and  $X=\overline{X}\setminus E$.
  If $\cP$ is a perverse sheaf on $X$, then $\chi(X,\cP)\ge 0$.  
\end{thm}

\begin{proof}

By work of Mumford et. al., we can find a smooth toroidal compactification $\overline{\Alevn{g_i}}$ of $\Alevn{g_i}$ such that boundary 
has simple normal crossings \cite[chap IV]{fc}.
Let
$$\overline{Y} = \overline{\Alevn{g_1}}\times \ldots  \times \overline{\Alevn{g_r}}$$
be a product of these  compactifications, and $D$ denote the simple normal crossings boundary divisor.
By a theorem of Faltings-Chai \cite[chap VI, thm 1.1]{fc}, we can find an nc log scheme
$(\overline{X}, E)$ with a morphism $f: (\overline{X},E) \to (\overline{Y},D)$ such that
\begin{enumerate}
    \item  $\overline{X}-E= X$ and $f|_{X}$ coincides with the product of canonical maps $\Alevn{g_i,n}\to \Alevn{g_i}$.
    \item  $f$ is toroidal \cite[chap VI, thm 1.13]{fc}, and therefore
log smooth by \cite[prop 3.4]{kato}.
\item $\Omega^1_{(\overline{X}, E)/(\overline{Y},D)}$ is the pull back of a Hodge  bundle on $\overline{Y}$.

\end{enumerate}

 The theorem now follows from Lemmas \ref{lem:pos}, \ref{lem:OmegaAgNef} and \ref{lem:logsmooth}, and Theorem \ref{thm:semipos}.

\end{proof}

\begin{thm}\label{thm:Agn}
    If $P$ is a perverse sheaf on $\sX= \sA_{g_1,n_1}\times\ldots \times\sA_{g_r, n_r}$, then
    $$\chi^{orb}(\sX,P)\ge 0$$
In particular,
    if $\sY\subset \sX$ is a smooth closed irreducible substack, then
    $$(-1)^{\dim \sY}\chi^{orb}(\sY)\ge 0$$
    Furthermore, we have strict inequality
    $$(-1)^{\dim \sY}\chi^{orb}(\sY)> 0$$
    if $\sY\subset \sA_g$ is proper.
    
\end{thm}

\begin{proof}
Since $\sX = \prod [\Alevn{g_i,n_i }/PSp_{2g_i}(\bbZ/N\bbZ)]$,  this follows
from Theorem \ref{thm:EulerSatake} (iii) and Theorems \ref{thm:chi}, \ref{thm:chiStrict} and \ref{thm:NAgn}.
    
\end{proof}

\begin{rem}
The theorem cannot be improved to a strict inequality in general, because
if $n>0$, then $\chi^{orb}(\sA_{g,n})=0$ (since $\Alevn{g,n}$ is an abelian scheme).
However, we do not know if there is 
an example of $X\subset \sA_g$ with $\chi(X)=0$.
\end{rem}

 Recall that a  morphism $f:X\to Y$ is called semismall if for all $m$
$$ 2m   \le  \dim Y-\dim\{y\in Y\mid \dim f^{-1}(y)=m\}$$

\begin{cor}\label{cor:semismall}
    If $\pi:X\to \sA_{g,n}$ is a semismall proper map from a smooth irreducible stack, then 
    $$(-1)^{\dim X}\chi^{orb}(X)\ge 0$$
\end{cor}

\begin{proof}
    By \cite[prop 4.2.1]{dm} $\pi_*\bbQ[\dim X]$ is perverse.
\end{proof}

We will give a concrete application of this below in Proposition \ref{prop:M3semi}.

\section{Key Euler Characteristic Inequality}\label{sec:key}

Given a variety $X$, we say that $X$ has the positive Euler characteristic property, or $PEC$-property, or simply that $X$ is $PEC$, if for every perverse sheaf $\cK \in \Perv(X)$, $\chi(X,\cK) \geq 0$. If $X$ comes equipped with a $\bbG_m$-action, we will use similar terminology in the setting of monodromic perverse sheaves (see \ref{subsec:monodromic}): $X$ has the $PEC_m$ property if for all $\cK \in Perv_m(X)$, $\chi(X,\cK) \geq 0$. In the monodromic setting, we shall say that $X$ has the $PEC_m[d]$-property if $\chi(X,\cK[d])\geq 0$ for all $\cK \in \Perv_m(X)$. Clearly $PEC_m[d]$ is the same as $PEC_m$ if $d$ is even, and $PEC_m[-1]$ otherwise.\\

By an admissible datum, we mean a $5$-tuple  $(X,Y,f,U,Z)$, where
 $f: X \rightarrow Y$ is a (not necessarily proper) birational morphism of irreducible quasiprojective varieties which is an isomorphism over $U \xhookrightarrow{j} Y$, and such that $Z = X \setminus U$ is an (effective) Cartier divisor on X. We say that the datum admits a $\bbG_m$-action if $\bbG_m$ acts on $X, Y, U, Z$ and $f$ is equivariant with respect to this action. The main result of this section is the following theorem.

\begin{thm}\label{thm:key1}
Let $(X,Y,f,U,Z)$ be an admissible  datum as above. Suppose that $Y$ and $Z$ have the $PEC$-property. Then $X$ also has the $PEC$-property. The analogous assertion holds in the monodromic setting, i.e.
when the datum admits a $\bbG_m$-action and $Y$ and $Z$ have the $PEC_m[d]$-property, then $X$  has the $PEC_m[d]$-property. 
\end{thm}

Below, we shall prove the theorem via induction on dimension, and henceforth assume that it holds for admissible data $(X',Y',f',U',Z')$ with $\dim(Y') <\dim Y=n$. In the first two subsections, we recall some preliminary material on Belinson's gluing construction for perverse sheaves, and their monodromic variants. The proof of the main theorem then proceeds via two reductions given in the following two subsections.

\subsection{Beilinson's gluing construction}\label{thm:beilinson}
We recall a basic result of Beilinson \cite{beilinson, morel, reich}.
We will follow the  notation from \cite{bbd} and write $f_*$ etc. for the {\em derived} direct image. Also, $\psi_f, \phi_f$ denotes derived nearby/vanishing cycle functors shifted by $[-1]$ (so as to preserve perverse sheaves). Let $\psi^u_f \subset \psi_f, \phi^u_f\subset \phi_f^u$ denote the maximal subfunctors on which the monodromy $T$ acts unipotently. The $can$ and $var$ maps restrict to these. Let $X$ be a variety with a regular function  $h:X \to \bbA^1$. Let  $Z := h^{-1}(0)$, $U= X-Z$, and let $i: Z \hookrightarrow X, j:U \hookrightarrow X$ denote the canonical inclusions. In this setting, one has the following fundamental result due to Beilinson.
 
\begin{thm}[Beilinson]
 There is an equivalence between the category of perverse sheaves
 $\Perv(X)$ and the  category 
 
\begin{equation*}
\begin{split}
B_h( U, Z) &= \{(K_U, K_Z, c, v)\mid K_U\in \Perv(U), K_Z\in \Perv(Z)\\
& c: \psi_h^u K_U\to  K_Z, v:K_Z\to \psi_h^u K_U, cv= I-T\} 
\end{split}
\end{equation*}
\end{thm}

In one direction, $K\in \Perv(X)$ goes to $(j^*K, \phi_h^u K, can, var)$.
The inverse is more complicated. Beilinson constructs a ``maximal extension" functor
$\Xi_h: \Perv(U)\to \Perv(X)$
which fits into exact sequences
\begin{equation}\label{eq:Xi}
0\to i_*\psi^u_h K_U\xrightarrow{\alpha}\Xi_h K_U\to j_* K_U\to 0
\end{equation}

$$0\to j_! K_U\to \Xi_h K_U\xrightarrow{\beta} i_*\psi^u_h K_U\to 0$$
The inverse sends an object $(K_U, K_Z, c, v)$ to the  cohomology of the complex

\begin{equation}\label{eq:beil}
K=H^0(i_*\psi^u_h K_U\xrightarrow{(\alpha,c)} \Xi_h K_U\oplus i_*K_Z\xrightarrow{(\beta,-v)} i_*\psi^u_h K_U)
\end{equation}
(where this is centered in degree 0).

\begin{rem}
We observe  that the first map in \eqref{eq:beil} is injective, and the second is surjective.
\end{rem}

\begin{lem}\label{lem:beil}
With the above notation,
    let $K\in \Perv(X)$, and let $(K_U, K_Z, c,v)$ be the corresponding object in $B_h(U,Z)$.
 Then
 $$\chi(K) = \chi( j_* K_U) + \chi(K_Z)-\chi(\psi_{ h}^u K_U)$$
\end{lem}

\begin{proof}
From   \eqref{eq:beil}, \eqref{eq:Xi}, and the above remark,  we obtain
    \begin{equation*}
\begin{split}
\chi(K) 
&= \chi(\ker ( \beta,- v)) - \chi(\psi_{ h}^u K_U)\\
 &= \chi(\Xi_{ h} K_U) + \chi(K_{Z}) - 2\chi(\psi_{ h}^u K_U)\\
 &= \chi( {j}_* K_U) + \chi(K_Z)-\chi(\psi_{ h}^u K_U)\\
\end{split}
\end{equation*}

\end{proof}
\subsection{Monodromic Sheaves}\label{subsec:monodromic}
Let $X$ be a variety over an algebraically closed field $k$ such that $char(k) = 0$. Suppose that $X$ comes equipped with a $\bbG_m$-action denoted by $\theta: \bbG_m \times X \rightarrow X.$ For $\lambda \in k^\times$, let 
$\theta_{\lambda} := \theta(\lambda,-) : X \rightarrow X$, and for $x \in X(k)$, let $w_x = \theta(-,x):\bbG_m \rightarrow X$. 
\begin{example}
A basic example of such data arises from a {\it cone}. More precisely, given a variety $X$ as above and a quasi-coherent graded $\cO_X$-algebra $\cA = \oplus_{i \geq 0} \cA_i$ of finite type generated in degree 1 such that $\cA_0 = \cO_X$, the usual $\Spec$ construction gives rise to an affine morphism $\Spec(\cA) \rightarrow X$. The scheme $\Spec(\cA)$ has a natural $\bbG_m$-action coming from the grading. Most of our examples will arise in this fashion.
\end{example}

A constructible sheaf (respectively complex of sheaves) $\cK$ on $X$ is {\it monodromic} if $w_x^*(\cK)$ (respectively the cohomology sheaves of $w_x^*(\cK)$) are locally constant for all $x \in X(k)$. We let $Perv_m(X)$ denote the full subcategory of monodromic perverse sheaves.

\begin{lem}(\cite{Verdier}, 3.2)
With notation as above, the following are equivalent:
\begin{itemize}
\item For all $\lambda \in k^{\times}$, $\theta_{\lambda}^*\cK \isom \cK$.
\item $\cK$ is monodromic.    
\end{itemize}
\end{lem}    
\begin{proof}
In the setting of {\it cones} and constructible sheaves, this is contained Proposition 3.2 of \cite{Verdier}. The same proof goes through in the current setting. Moreover, the case of complexes can be reduced to sheaves by passing to cohomology sheaves.
\end{proof}   

Let $X$ be a variety equipped with a $\bbG_m$-action, and $h: X \rightarrow \bbA^1$ a morphism which is equivariant for the standard $\bbG_m$-action on $\bbA^1$. Then $Z:= h^{-1}(0)$ and $U=X\setminus Z$ are invariant under $\bbG_m$. In this case, we may consider the categories $\Perv_m(X), \Perv_m(U),$ and $\Perv_m(Z)$.

\begin{thm}\label{thm:beilisonmonodromic}
With notation as above, $\Perv_m(X)$ is equivalent to the full subcategory $B_{h,m}(U,Z) \subset B_{h}(U,Z)$ consisting of objects $(K_U, K_Z, c,v)$ such that $K_U$ and $K_Z$ are monodromic.
\end{thm}
\begin{proof}
First, note that $j^*\cK$ is monodromic if $\cK$ is monodromic. This follows directly from the previous lemma. Moreover, the vanishing cycles along $h$ of a monodromic sheaf is also monodromic. This follows from the previous lemma and the fact that vanishing cycles commutes with pull-back along acyclic morphisms. Since smooth morphisms and isomorphisms are such morphisms the result follows from the previous lemma (apply the previous remarks to $\theta_\lambda$).
\end{proof}

\begin{rem}\label{rem:mondromicundershreik}
Suppose that $f: X \rightarrow Y$ is compatible with $\bbG_m$-actions. Then $f_{!}$ preserves monodromic sheaves. This follows from the previous lemma and the fact that $f_{!}$ is compatible with base change. Namely, one has the cartesian square:
$$
\xymatrix{
X \ar[r]^{\theta_{\lambda}} \ar[d]^f & X \ar[d]^f \\
Y \ar[r]^{\theta_{\lambda}}  & Y}
$$
In particular, this holds for $f_*$ when $f$ is  proper, and $j_!$ where $j$ is an open immersion.
\end{rem}

\subsection{Proof of Theorem \ref{thm:key1}: Preliminary reduction}

In this subsection, we assume the existence of a non-constant regular function $g:Y\to \bbA^1$ such 
that  $Z = h^{-1}(0)$, where $h= g\circ f$. Let $D = f(Z)= g^{-1}(0)$. 
We prove Theorem \ref{thm:key1} under these conditions. Below, $\tilde{j}: U \hookrightarrow X$ denotes the natural inclusion. \\

We begin by recalling a result of Laumon (\cite{laumonEuler}). For a variety, let $\rK_0(X)$ denote the Grothendeick group of constructible sheaves. Given a morphism $f: X \rightarrow Y$, we have induced morphisms $Rf_*, Rf_!: \rK_0(X) \rightarrow \rK_0(Y)$. We also have $f^{-1}: \rK_0(Y) \rightarrow \rK_0(X)$. For $\cF$ in the derived category of constructible sheaves on $X$, $[\cF]$ denotes its class in $\rK_0(X).$

\begin{thm}(\cite{laumonEuler})\label{thm:laumon}
With notation as above: $[Rf_*(\cF)] = [Rf_!(\cF)]$ in $\rK_0(Y)$ for all $[\cF] \in \rK_0(X)$.
\end{thm}

We have the following easy consequence which will be used below. 

\begin{lem}\label{lem:laumon}
Let $f: X \rightarrow Y$ be a morphism and $D \xhookrightarrow{i} Y$ be a closed subscheme. Let $Z = f^{-1}(D) \xhookrightarrow{\tilde{i}} X$ and $f_Z: Z \rightarrow D$ denote the restriction of $f$ to $Z.$ Then for any $\cF \in \rK_0(X)$ we have $[i^{-1}Rf_*\cF] = [Rf_{Z,*}(\tilde{i}^{-1}\cF)] \in \rK_0(D)$. 

\end{lem}

Note that if $f$ is proper, then the lemma is an immediate from proper base change. 

\begin{proof}
We have $[Rf_*(\cF)] = [Rf_{!}\cF]$ in $\rK_0(Y)$ by Theorem \ref{thm:laumon}. It follows that $[i^{-1}Rf_*\cF] = [i^{-1}Rf_!\cF]$ in $\rK_0(D)$. By base change for $f_{!}$, we have 
$[i^{-1}Rf_!\cF] = [Rf_{Z,!}(\tilde{i}^{-1}\cF)]$. Another application of Theorem \ref{thm:laumon} gives the desired conclusion.
\end{proof}

\begin{lem}\label{lem:eulerchardiffpos}
Let $(X,Y,f,U,Z)$ be an admissible datum such that $Y$ has the $PEC$-property.
 If $K_U\in \Perv(U)$, then
 $$\chi(\tilde j_* K_U) -\chi(\psi_{h}^u K_U)\ge 0.$$
\end{lem}

\begin{proof}
Without loss of generality, we can assume that the ground field $k=\bbC$.
 Given $a\in \bbN$, let $L_a$ be the pull back to $U$ of the local system on $\bbC^*$ with monodromy given by the $a\times a$ Jordan
 block
 $$
\begin{pmatrix}
 1 & 1& 0 & 0&\ldots\\ 0 &1 & 1 &0 &\ldots\\ 0&0 & 1 &1 &\ldots\\
  && \ldots &&
\end{pmatrix}
 $$
 We choose (and fix) $a\gg 0$. We may  replace $K_U$ by $K_U\otimes L_a$, since the expression $\chi(\tilde j_* K_U) -\chi(\psi_h^u K_U)$
 just gets multiplied by $\operatorname{rank} L_a$. It follows from \cite[prop 3.1]{morel}, that the 
  monodromy of $\psi_g^u K_U$ and 
  $\psi_h^u K_U$ is now trivial. 
Let $M= j_{!*} K_U$. and $\tilde M=   \tilde j_{!*} K_U$.
By \cite[prop 4.7]{reich}, 
$\tilde M$ corresponds to
$$(\tilde M_U= K_U, 0,0,0)\in B_h(U, Z)$$
and 
$M$ corresponds to
$$( M_U= K_U,  0,0,0)\in B_g(U, D).$$
Now we note that:
\begin{enumerate}
\item Since $M \in Perv(Y)$ and $Y$ has PEC, we have $\chi(Y,M) \geq 0$. On the other hand, by Lemma \ref{lem:beil}, this is
$\chi(Y,j_*K_U) - \chi(Y,\psi^u_g(K_U))$ since $\phi^u_g(K_U) = 0$. 
\item Since $f_*(\tilde{j}_*K_U) = j_*K_U$, $\chi(X,\tilde{j}_*K_U) = \chi(Y,j_*K_U)$.
\end{enumerate}

It remains to show that $\chi(Y,\psi^u_g(K_U)) = \chi(X,\psi^u_h(K_U))$. Note that we are in the setting where $\psi^u_{g}(K_U) = i^{-1}j_*K_U$ and similarly for $\psi_{h}^u(K_U)$ (i.e. vanishing cycles are zero). 
By Lemma \ref{lem:laumon}, we have $[Rf_{Z,*}\psi_{h}^u(K_U)] = [Rf_{Z,*}(\tilde{i}^{-1}\tilde{j}_*K_U)] = [i^{-1}Rf_*(\tilde{j}_*K_U)] = [i^{-1}j_*K_U] = [\psi^u_g(K_U)]$.
\end{proof}

Combining this with Lemma \ref{lem:beil}, we obtain:

\begin{cor}\label{cor:affinePECcor}
Consider an admissible datum $(X,Y,f,U,Z)$ as above and suppose further that there is a function $g: Y \rightarrow \bbA^1$ such that $(g \circ f)^{-1}(0) = Z.$ If $Y$ and $Z$ satisfy the $PEC$-property, then so does $X$. In the presence of $\bbG_m$-actions, the analogous statement holds if $Y$ and $Z$ have the $PEC_m[d]$-property.

\end{cor}

\subsection{Proof of Theorem \ref{thm:key1}: Second reduction.}

In this subsection, we prove Theorem \ref{thm:key1} in the setting where we have a function $g: Y \rightarrow \bbA^1$ such that $f(Z) \subseteq g^{-1}(0).$ We shall make this assumption for the remainder of this subsection.

\begin{lem}\label{lem:posforunionofcartier}
Let $Z$ be a reduced effective Cartier divisor on a variety $X$. Suppose $Z = Z_1 \cup Z_2$ where each $Z_i$ is a union of components of $Z$ (with no components in common). If the $PEC$-property holds for each $Z_i$, then it also holds for all objects of $Perv(Z)$. The analogous assertion holds in the monodromic setting.
\end{lem}

\begin{proof}  Let $P \in \Perv(Z,S)$. Then $\chi(Z,P)= \chi(Z_1, P)+ \chi(Z_2, j_{!}P)$, where $j:Z_2-Z_1\to Z_2$ is inclusion. On the other hand, under the given hypotheses $Z_1 \cap Z_2$ is an effective Cartier divisor on $Z_2$, and it equals $Z_2 \setminus (Z_2 \setminus Z_1)$. It follows that $j_!P$ is perverse, and the result now follows. In the monodromic setting, we note that the same argument works because $j_!$ preserves monodromic perverse sheaves. 
\end{proof}

\begin{lem}\label{lem:Keylemmaaffine}
\-
\begin{enumerate}
\item With notation and hypotheses as in  theorem \ref{thm:key1}, suppose furthermore that there exists a function $g:Y \rightarrow \bbA^1$ such that $g^{-1}(0) \supseteq f(Z)$. Then the conclusion of the theorem holds. 
\item Suppose that the tuple $(X,Y,f,U,Z)$ and $h$ as above all come equipped with a $\bbG_m$-equivariant action. If $Y$, $Z$ satisfy the $PEC_m[d]$ property, then so does $X$.
\end{enumerate}
\end{lem}
\begin{proof}
For (1), let $D_Y := g^{-1}(0)$ and $D_X := h^{-1}(0)$ where $h := g \circ f: X \rightarrow \bbA^1$. Note that $f(D_X)= f(f^{-1}(D_Y)) = D_Y$. Let $V = Y \setminus D_Y \subset U$. It follows that this gives rise to an admissible datum $(X,Y,f,D_X,V)$. Moreover, we claim that for every perverse sheaf $K \in \Perv(D_X)$, $\chi(D_X,K) \geq 0$. In order to see this, first note that the analogous assertion holds for $D_Y$ since it holds for $Y$ (and push-forward from a closed subscheme is exact for the perverse t-structure). We may write $D_X = D'_X \cup Z$ as a union of Cartier divisors. Similarly, set $D_Y = D'_Y \cup f(Z)$. Restricting $f$ to $D'_X$ gives rise to an admissible tuple $(D'_X,D'_Y,f,Z \cap D'_X,U\setminus V)$. We may assume by induction on dimension of $Y$ (or equivalently $X$) that the Euler characteristic of every perverse sheaf on $D'_X$ is non-negative. Note that this property holds for both $D'_Y$ and $Z \cap D'_X$ since it holds for $Y$ and $Z$. Now by Lemma \ref{lem:posforunionofcartier} above, we may conclude that the Euler characteristic of every perverse sheaf on $D_X$ is non-negative, as claimed. Now for the tuple $(X,Y,f,V,D_X)$, the result follows from Corollary \ref{cor:affinePECcor}. Therefore, this proves the desired property for all perverse sheaves on $X$. 

In the monodromic setting of (2),  the same proof goes through. 

\end{proof}

\subsection{The general case}
In order to complete the proof, we use deformation to the normal cone. Given a pair $Z \subset X$ with $Z$ a closed subvariety, let $C(Z/X)$ denote the normal cone of $Z$ in $X$, and let $C^*(Z/X) := C(Z/X) \setminus Z$, $U := X \setminus Z$. The deformation to the normal cone construction gives rise to a diagram:
$$
\xymatrix{
U \ar[r] & X & Z \ar[l] \\
U \times \bbG_m \ar[r] \ar[u] \ar[d] & V^*(Z/X) \ar[u] \ar[d] & \ar[l] \ar[u] \ar[d] C^*(Z/X) \\
X \times \bbG_m \ar[r]\ar[d] & V(Z/X) \ar[d]  & \ar[l] \ar[d]C(Z/X) \\
 \bbG_m \ar[r]  & \bbA^1 & \ar[l] \{0\}. }
$$
We remind the reader that $V(Z/X)$ is obtained as follows. First, consider the blow-up $B(Z/X) := Bl_{Z \times \{0\}}(X \times \bbA^1)$. Then set 
$V(Z/X):= B(X/Z) \setminus (Bl_Z(X))$ and $V^*(Z/X)= V(Z/X)- Z\times \bbA^1 .$
Algebraically, one obtains $V(Z/X)$ as follows. Let $\cI \subset \cO_X$ denote the ideal of $Z$.
Then $V(Z/X) ={\bf Spec} (R_Z)$, where $R_Z := \bigoplus_{n\in \bbZ} \cI^{n}t^{-n} \subset \cO_X[t,t^{-1}]$ is the (extended) Rees algebra.
We see from this that the natural projection map $p: V(Z/X) \rightarrow X$ is flat. If $Z$ is a Cartier divisor, then all fibers are smooth of dimension 1, and hence $p$ is smooth.
 We note that the bottom three rows of the previous diagram come equipped with $\bbG_m$-actions, and the morphisms in the diagram respect these actions.
 These constructions are functorial in the following sense.

\begin{lem}\label{lem:defconefun}
 If $Z\subset X$ and $Z'\subset X'$ are closed subvarieties, and if $f:X\to X'$
 is a   morphism such that $Z= f^{-1}Z'$, then there is a $\bbG_m$-equivariant morphism $F:V(Z/X)\to V(Z'/X')$
 which induces a map of the corresponding diagrams above. 
 
\end{lem}

\begin{proof}

 Suppose $Z= f^{-1}Z'$, then $f$ induces a   map
 $F_1:B(Z/X)\to B(Z'/X')$ between the blow ups 
such that $F_1^{-1}(Bl_{Z'}X'\times \{0\})= Bl_Z X\times \{0\}$.  This  restricts to give $F:V(Z/X)\to V(Z'/X')$. 
\end{proof}

We also need the following.

\begin{lem}\label{lem:VhasPEC}
Let $Z\subset X$ be an effective Cartier divisor. Then
$X$ has the $PEC$-property if and only if $V(Z/X)$ has the $PEC_m[-1]$-property.

\end{lem}

\begin{proof}
    Since $Z \hookrightarrow X$ is a Cartier divisor, then  $p: V(Z/X) \rightarrow X$ is smooth of relative dimension 1. It follows that if $\cK \in \Perv(X)$, then $p^*(\cK)[1]$ is a monodromic perverse sheaf on $V(Z/X)$. Moreover, the adjunction $\cK \rightarrow Rp_*p^*\cK$ is an isomorphism. The lemma easily follows.
\end{proof}

\begin{proof}[Proof of Theorem \ref{thm:key1}]
 Suppose we are given an admissible datum $(X,Y,f,U,Z)$ such that $Y$ and $Z$ have the $PEC$ property. 
 Since $Y$ is quasiprojective, we can choose  a Cartier divisor $D\supseteq f(Z)$.
Let us  write $f^{-1}D= Z\cup Z'$, where $Z'$ is the union of components $Z_i'$ not contained in $Z$.
Then  $(Z_i', f(Z_i'), f|_{Z_i'}, U\cap Z_i', Z_i'\cap Z)$ are admissible data satisfying the assumptions of the theorem. By induction and Lemma \ref{lem:posforunionofcartier}, we find that
$Z'$ has the $PEC$ property. Since $Z$ also has the $PEC$ property, we can conclude that
so does $Z\cup Z'$ by Lemma \ref{lem:posforunionofcartier}. Therefore $C(Z\cup Z'/X)$ has the $PEC_m[-1]$ property by Lemma \ref{lem:VhasPEC}.
This  Lemma also shows that  $V_Y := V(D/Y)$ has the $PEC_m[-1]$-property. Let 
 $V_X= V(Z\cup Z'/X)$. The map $f$ induces a morphism
  $F:V_X\to V_Y$ by Lemma \ref{lem:defconefun}.
Therefore we have
 admissible datum $(V_X,V_Y,F,U \times \bbG_m, Z'')$, where $Z'' = C(Z\cup Z'/X) \cup (Z\cup Z') \times \bbA^1$. Lemma \ref{lem:posforunionofcartier} shows that $Z''$ has the $PEC_m[-1]$ property.
Moreover, we have $g^{-1}(0) \supseteq F(Z'')$, where
 $g: V_Y  \rightarrow \bbA^1$ is the projection. In particular, we are in the setting of Lemma \ref{lem:Keylemmaaffine}. It follows that $V_X$ has the $PEC_m[-1]$-property, and therefore $X$ has the $PEC$-property by Lemma \ref{lem:VhasPEC}.

\end{proof}

\section{Extended Torelli map}\label{sec:torelli}

Let $\overline{\sM_g}$ (respectively $\sM_g$) denote the stack parameterizing stable (respectively smooth) curves of genus $g$ \cite{dm}.
Recall that a  stable curve $C$  is of compact type if its dual graph is a tree, or equivalently if  the identity component of its Picard scheme $\Pic^0(C)$
 is an abelian variety.  Let 
$$\sM_g\subset\sM_{g}^c\subset {\overline{\sM_{g}}}$$
be the open subset parameterizing  curves
of compact type.

 By a  theorem of Mumford and Namikawa \cite[\S9D]{namikawa} the Torelli map $\tau:\sM_g\to \sA_g$ extends to a map
$\tau: \overline{\sM_g}\to \overline{\sA_g}$ to the toroidal
compactification associated to the second Voronoi decomposition. Alexeev \cite{alexeev1, alexeev2} gave a modular interpretation of both the target
and the map.
 On $\sM_g^c$, $\tau$ simplifies to
$$\tau([C]) = [\Pic^0(C)]  =[\prod_i (\Pic^0(C_i)] \quad \text{(as principally polarized abelian varieties})$$
where $C_i$ are the irreducible components of $C$.
An immediate consequence of  properness of $\overline{\sM_g}$ and the theorem of Mumford-Namikawa is:

\begin{lem}
 The map $\tau$ restricts to a proper map
$\tau:\sM_g^c\to \sA_g $. Therefore $\tau(\sM^c_g)$ is the closure of $\tau(\sM_g)$
in $\sA_g$ 
\end{lem}

We will give a more direct proof of properness below.
Given $N>0$, let ${}_N\sM_{g}$ denote the stack of smooth curves with level $N$ structure. Let
$\overline{{}_N\sM_g}$ (respectively  ${}_N\sM^c_{g}$) denote the normalization of $\overline{\sM_g}$ (respectively  $\sM^c_{g}$) in the function field of 
${}_N\sM_g$. The stack  ${{}_N\sM^c_g}$  can be identified with the moduli stack of stable curves of compact type with level $N$ structure (by contrast
$\overline{{}_N\sM_g}$ does not have a good modular interpretation). Thus we can lift  $\tau$ to a map
$\tau: {}_N\sM^c_{g}\to {}_N\sA_{g}$.

Let $\overline{\sM_{g,n}}$ be the Deligne-Mumford stack of genus $g$ stable
curves with $n$ base points \cite{knudsen}. A geometric point of this  stack corresponds
a semistable curve $(C, p_1,\ldots, p_n)$ of arithmetic genus $g$ with $n$ smooth base points such that the total number of base and double points on any smooth rational component is at least three.  We refer to this is as an $n$-pointed stable curve.
When $n>0$, there exists a morphism $\pi_n: \overline{\sM_{g,n}}\to \overline{\sM_{g, n-1}}$ called a contraction. For a  curve without rational components, $\pi_n$ simply forgets $p_n$.  Given integers $g= g_1+g_2$, and a partition of $\{1,\ldots, n\}$ into
two subsets of cardinality $n_1$ and $n_2$, Knudsen also constructs a gluing map
$$gl:\overline{\sM_{g_1,n_1+1}}\times \overline{\sM_{g_2, n_2+1}}\to \overline{\sM_{g, n}}$$
which takes  $((C, p_1,\ldots, p_{n_1+1}), (C', p'_1,\ldots, p'_{n_2+1}))$ to the curve $C\coprod C'/(p_{n_1+1}\sim p'_{n_2+1})$
with the remaining base points  $p_1,\ldots,p_{n_1},p_{1}',\ldots, p_{n_2}'$ enumerated appropriately. 

We recall two basic facts  regarding these constructions (see \cite{knudsen}).
\begin{enumerate}
\item First, the contraction map is equivalent to the universal curve over $\sM_{g,n-1}$ (see Lemma \ref{lem:unicurve}). By {\it equivalent} we mean there is an isomorphism of stacks over $\sM_{g,n-1}$.
\item Second, the map $gl$ is finite.
\end{enumerate}
The forgetful map $\overline{\sM_{g,n}}\to \overline{\sM_{g}}$ is the composition of contractions.
Let ${\sM_{g,n}}$ and ${\sM^c_{g,n}}$ denote the preimages of $\sM_g$ and $\sM_g^c$ respectively.
The complement $\sM^c_{g,n}- \sM_{g,n}$ is a union of divisors, each of which can be described as the image of an aforementioned gluing map (intersected with $\sM_{g,n}^c$).

\begin{lem}\label{lem:unicurve}
 The contraction map $\sM^c_{g,n+1}\to \sM^c_{g,n}$ is equivalent to the universal curve over  $\sM^c_{g,n}$.
\end{lem}

\begin{proof}
 Given a scheme $S$, an $S$-valued point of $\overline{\sM_{g,n}}$ (respectively $\sM^c_{g,n+1}$) is a stable curve $\cC\to S$ with $n$ sections $\sigma_1,\ldots, \sigma_{n}$,  such that the geometric fibres are $n$-pointed genus $g$ stable curves (of compact type). The universal curve
 $\overline{\sZ_{g,n}}\to \overline{\sM_{g, n}}$ (respectively  $\sZ^c_{g,n}\to \sM^c_{g, n}$) is the stack whose $S$-points are stable curves $\cC\to S$ with $n+1$  sections  $\sigma_1,\ldots, \sigma_{n+1}$,
 such that $(\cC\to S,\sigma_1,\ldots, \sigma_n)$ defines a point of $\overline{\sM_{g,n}}$ (respectively $\sM^c_{g,n}$).
 There is an obvious functor $\overline{\sM_{g,n+1}}\to  \overline{\sZ_{g,n}}$, which sends to $ \sM^c_{g, n+1}$ to $\sZ^c_{g,n}$.
 Knudsen \cite[\S2]{knudsen} defines an inverse $s:\overline{\sZ_{g,n}}\to  \overline{\sM_{g,n+1}}$ called stabilization.
When $S=\Spec k$, then roughly
 $$s(\sC, \sigma_1,\ldots, \sigma_{n+1}) 
 =
\begin{cases}
 (\cC, \sigma_1,\ldots, \sigma_{n+1}) &\text{if this is an $(n+1)$-pointed stable curve}\\
 (\cC',\sigma_1',\ldots, \sigma_{n+1}') &\text{if $\sigma_{n+1}$ is a singular point or  another section} \\
\end{cases}
$$
where $\cC'$ is  obtained by adding a $\bbP^1$, such that on the dual graph  it is inserted  either as a pendant vertex or on an edge.
Since this operation takes trees to trees, it follows that $s(\sZ^c_{g,n})\subset  {\sM^c_{g,n+1}}$.
 Thus $\sM^c_{g,n+1}$ is equivalent to $\sZ^c_{g,n}$.
\end{proof}

\begin{prop}\label{prop:extTorelli}
The morphism $\tau: \sM_{g,n}\to {\sA_{g,n}}$ given on geometric points by
$$\tau[(C,p_1,\ldots, p_n)] = [(\Pic^0(C), (2g-2)p_1-K, \ldots ,(2g-2)p_n-K)],$$
where $K$ is the canonical divisor class, is  well defined and injective on points.
 It extends to  a proper morphism (denoted by same letter) $\tau: \sM^c_{g,n}\to {\sA_{g,n}}$.
 Furthermore, this lifts to a morphism  $\tau: {}_N\sM^c_{g,n}\to {}_N{\sA_{g,n}}$ for any $N$,
 where ${}_N\sM^c_{g,n} = {}_N\sM^c_{g}\times_{\sM^c_{g}} \sM^c_{g,n}$
\end{prop}

\begin{proof}
For existence,
it suffices to treat the case where $N\ge 3$, since the other cases can be handled by taking quotients.
 Then $M= {}_N\sM^c_{g,n}$ is fine with a universal family $\cC= {}_N \sZ_{g,n}$ over it. 
 The curve $\cC$ carries universal sections $\sigma_1, \ldots, \sigma_n$.
 The relative Picard algebraic space   $\Pic(\cC/ M)$ exists by \cite[chap 9.4, thm 1]{blr}.
 Furthermore, this is a disjoint union of subschemes $P^d= \Pic^d(\cC/ M)$ parameterizing line bundles $\cL$
 of degree $d$ on the fibers, where we define the the degree to be $\chi(L)-g+1$. Note that $P^0\to M$ is a $g$-dimensional abelian scheme with level $N$-structure,
 and $P^d$ is a torsor over $P^0$. When $d= 2g-2$, there is a canonical section given by the class  of $\omega_{\cC/M}$, which allows us to
 identify $P^{2g-2}$ with $P^0$. The class of $(2g-2)\sigma_i$ defines a section of $\Pic(\cC/M)$. 
 In fact, by restricting to ${}_N \sM_{g,n}$, we can see that $(2g-2)\sigma_i$ lies in  $P^{2g-2}$. We can view these as sections of $P^0$ by the above
 identification. Therefore, the universal property of ${}_N \sA_{g,n}$  gives 
 map $\tau: M \to {}_N{\sA_{g,n}}$. On points of  $M$, it is clearly described  (with some abuse of notation) as
 $$\tau[(C,\sigma_1,\ldots, \sigma_n)] = [(\Pic^0(C), (2g-2)\sigma_1-\omega_C, \ldots, (2g-2)\sigma_n-\omega_C)],$$
 Abel's theorem plus Torelli's theorem
 shows that $\tau$ is  injective on $\sM_{g,n}$.

 To check  properness, we use the valuative criterion for stacks (\cite[thm 4.19]{dm}, or \cite{stacks-project}, \href{https://stacks.math.columbia.edu/tag/0CLY}{Tag 0CLY}, \href{https://stacks.math.columbia.edu/tag/0CLK}{Tag 0CLK}
, \href{https://stacks.math.columbia.edu/tag/0CLG}{Tag 0CLG}).  Let $R$ be a discrete valuation ring with  generic point $\eta$, and suppose that we are given a $2$-commutative diagram
  $$
 \xymatrix{
 \eta\ar[r]\ar[d] & \sM^c_{g,n}\ar[d]^>>>>>{\tau}\ar@{}[r]|{\subset} & \overline{\sM_{g,n}} \\ 
  \Spec R\ar[r]\ar@{-->}[ru]^>>>>>{h}\ar@{..>}[rru]_>>>>>>{f} & \sA_{g,n} & 
}
  $$
  with solid arrows.  Then after replacing $R$ by a finite extension, we will construct an arrow $h$ making the new diagram $2$-commute.\footnote{If $\eta = \Spec(K)$, then `after a finite extension of $R$' means we pass to a finite extension $K'$ of $K$ and a valuation ring $R' \subset K'$ dominating $R$.}
 Since $\overline{\sM_{g,n}}$ is proper \cite{knudsen}, we do get an arrow $f:\Spec R\to \overline{\sM_{g,n}}$  (after a finite extension), making the resulting diagram
 commute.
 In geometric language, we are given a $g$-dimensional principally polarized abelian scheme $\cA\to \Spec R$, with $n$ sections $\alpha_i$.
 The map $f$ gives  a stable  $n$-pointed genus $g$ curve $(\cC\to \Spec R, \sigma_i)$ such that
 $$(\cA, \alpha_i)_\eta \cong \tau(\cC_\eta, \sigma_{i,\eta})= (\Pic^0(\cC/\Spec R),  (2g-2)\sigma_i-\omega_\cC)_\eta$$
 By the theory of  N\'eron models \cite[chap 9.5, thm 4]{blr} , we have 
 $\cA\cong \Pic^0(\cC/\Spec R)$. It follows that $\cC$ is of compact type. Since $\cA/\Spec R$ is separated, we have an isomorphism
 $$(\cA, \alpha_i) \cong (\Pic^0(\cC/\Spec R),  (2g-2)\sigma_i-\omega_\cC)$$
  Therefore the image of $f$ lies in  $\sM^c_{g,n}$, and moreover $h=f$ makes the diagram commute. The extension to the level $N$ case is straightforward and left to the reader.
\end{proof}

Let us say a pointed curve, corresponding to a point of $\sM^c_{g,n}$,
is of special compact type if the dual graph has the property  that no two genus zero vertices are joined by an edge. In particular, the dual graph has no rational trees. We now let $\sM_{g,n}^{sc} \subset \sM_{g,n}^{c}$ denote the  substack of curves of special compact type. 
This can be seen to be an open substack.

Note that $\sM_{0,n}^{sc} = \sM_{0,n}$, and $\sM_{1,1}^{sc} = \sM_{1,1}$. In general, the boundary divisors $\sM_{g_1,n_1 +1}^{c} \times \sM_{g_2,n_1 +1}^{c}$ of $\sM^{c}_{g,n}$ restricted to $\sM_{g,n}^{sc}$ give the boundary divisors $\sM_{g_1,n_1+1}^{sc} \times \sM_{g_2,n_2+1}^{sc}$ (where we have $g_{1} + g_2 = g$ and $n_1 + n_2 = n$). For example, the boundary of $\sM_{1,2}^{sc}$ is the divisor $\sM_{1,1} \times \sM_{0,3}$. We may consider analogs with the level $N$-structure ($g > 0$) defined as before.

\begin{lem}\label{lem:imageclosed}
We claim that $\tau(\sM_{g,n}^{sc}) = \tau(\sM_{g,n}^{c})$,
and therefore this set  is closed in $\sA_{g,n}$ (since $\tau$ is proper on $\sM_{g,n}^{c}$).

\end{lem}

\begin{proof}
    We may check the claim on geometric points. If we have a curve with a rational chain in the dual graph, we can contract it to a single $\mathbb{P}^1$ attached to a higher genus curve or two different higher genus curves. This object is in $\sM_{g,n}^{sc}$ and has the same image under $\tau$.

\end{proof}

\section{Euler characteristics of perverse sheaves on $\sM_{g,n}$}

 Let $j:\sM_g\subset \tau(\sM_g^c)=\overline{\tau(\sM_g)}$ denote the inclusion, where $\tau$ is discussed in the previous section.
 We start with a warm up to the main result.
 
\begin{prop}
    If $P$ is a perverse sheaf on $\sM_g$, then $\chi^{orb}(j_{!*}P)\ge 0$.
\end{prop}

\begin{proof}
 This follows from theorem \ref{thm:Agn}.
\end{proof}

When $g=2$, the map $\tau$  is finite and surjective, and therefore semismall.
So corollary \ref{cor:semismall} applies. 
When $g>2$, the map is not finite, however:
\begin{prop}\label{prop:M3semi}
    If $g=3$ then $\tau$ is semismall. Therefore
    $\chi^{orb}(\sM_3^c)\ge 0$
\end{prop}

\begin{proof}
    When $g=3$, we claim that
  the fibres of $\tau$ are at most one dimensional, and these lie over the codimension $2$ subset of $\sA_3$ parameterizing reducible abelian varieties. 
  To prove the claim, observe that a singular curve $C$
  in $ \sM_3^c$ is either 
  \begin{enumerate}
      \item[(a)] a union $C=C_1\cup C_2$ of an genus $1$ curve and a genus $2$ curve meeting at a point $p$,
      \item[(b)]   a union $C=C_1\cup C_2\cup C_2$ of three genus $1$ curves where the dual graph is chain,
      \item[(c)] or a union  $C= C_0\cup C_1\cup C_2\cup C_3$, where $C_1,C_2,C_3$ have genus $1$,  $C_0=\bbP^1$, and the dual graph is a trivalent tree with $C_0$ as  the root.
  \end{enumerate}
 In  case (a), $C_1$ can be taken to be an elliptic curve with $p$ as its origin. Since
 $\tau([C])= [\Pic^0(C_1)\times \Pic^0(C_2)]$,  the map forgets $p\in C_2$.
  Then we may identify the fibre  $\tau^{-1}(\tau([C]))$ with $C_2$. In case (b),
  we take $C_1\cap C_2$  as the origins of $C_1$, $C_2$, and $C_2\cap C_3$ as the origin of $C_3$.
The point $p\in C_2\cap C_3 $ is distinct from $0$ in this case; we treat (c)
as the limiting case where $p\to 0$. Then $\tau^{-1}(\tau([C]))$ with $C_2$.
   The claim follows from this, and it implies that  $\tau$ is also semismall, when $g=3$. The inequality follows from corollary \ref{cor:semismall}.
\end{proof}

\begin{rem}
The preceding proposition now implies our Theorem \ref{thm:mainthm} for $\sM_{3}$ since $\sM_3 \hookrightarrow \sM_{3}^g$ is an affine embedding. However, the map $\sM_g^c\to \sA_g $ is not semismall when $g>3$, so in general one requires a different strategy. For general $\sM_{g,n}$ (and their products) we shall use an inductive argument via the partial compactifications $\sM_{g,n}^{sc}$ (and their products).
\end{rem}

As noted in the introduction, we shall prove Theorem \ref{thm:intro3} via induction. The following proposition provides a base case for the induction.

\begin{prop}
 Let $\sM= \sM_{0, n_1}\times\cdots \times \sM_{0, n_r} \times {}_N\sA_{g_s,m_s} \times \cdots \times {}_N\sA_{g_1,m_1} $, where $g_i > 0$, $m_i \geq 0$, $n_i > 2$ and $N \geq 3$.
   If   $\cK$ is perverse sheaf on $\sM$, then
    $$\chi(\sM,\cK)\ge 0$$
\end{prop}
\begin{proof}
Write $\sM$ as a product $M \times A$ where $M$ is a product of factors of type $\sM_{0,n}$ and $A$ is a product of factors of type ${}_N\cA_{g,n}$. Since $\sM_{0, 3}$ is a point, we may assume that  $n_i>3$. We note that $\sM_{0,4} \cong  \bbP^1-\{0,1,\infty\}$.
 If $E= 0+1+\infty$, then $\Omega_{\sM_{0,4}}^1(\log E)$ is nef. Therefore, the proposition holds for any product $(\sM_{0,4})^l$  by
 lemmas \ref{lem:pos} and \ref{lem:prodnef}.
 We have an  open immersion $j:\sM_{0,n}\to (\sM_{0,4})^{n-3}$ sending
 $$(C, p_1,\ldots, p_n)\mapsto ((C,p_1, p_2, p_3, p_i))_{i=4,\ldots n}$$
The complement of $\im(j)$ is given by a (big) diagonal, which is  a divisor. Therefore  $j$ is affine. It follows that the product map
$$\sM_{0, n_1}\times\ldots \times \sM_{0, n_r}\xrightarrow{J}  (\sM_{0,4})^{\sum n_i-3r} =: X$$
is also an open affine embedding. By Theorem \ref{thm:NAgn},  there is a compactification $\overline{A}$ of $A$ such that the log tangent bundle is nef. It follows by Lemma \ref{lem:prodnef} that the log cotangent bundle of $X \times \overline{A}$ is nef, and therefore by Lemma \ref{lem:pos} the conclusion of the proposition holds for $X \times A$. Now $J \times id$ is an open affine embedding
of $M \times A$ into $X \times A$. Therefore $ (J \times id)_* P$ is perverse when $P\in \Perv(M \times A)$. Therefore
$\chi^{orb}(\sM, P)\ge 0$.

\end{proof}

\begin{thm}\label{thm:Mg}
 Let $\sM= \sM_{g_1, n_1}^{sc}\times\ldots \times \sM_{g_r, n_r}^{sc}$, where $2g_i-2 +n_i\ge 3$.
   If   $\cK$ is perverse sheaf on $\sM$, then
    $$\chi^{orb}(\sM,\cK)\ge 0$$
    
\end{thm}

\begin{proof}
For a product $\sM =\prod\sM_{g_i,n_i}^{sc}$ as in the theorem, let $H(\sM) = \max_{i} d_i$ and $S(\sM) = \Sigma_{i} d_i$ where $d_i = 3g_i-3 + n_i$ (the dimension of the given component). We consider pairs of natural numbers $(H,S)$ ordered lexicographically: $(H,S) < (H',S')$ if $H < H'$ or $H = H'$ and $S< S'$. We shall proceed by induction with respect to this ordering.

Fix $N\ge 3$, and let $\sM= {}_N\sM_{g_1, n_1}^{sc}\times\ldots \times {}_N\sM_{g_r, n_r}^{sc}$. Here, we set ${}_N\sM_{0, n}^{sc} := {}_N\sM_{0,n}$. 
It is enough to prove that
\begin{equation}\label{eq:Mg}
\chi(\sM,K)\ge 0 
\end{equation}
for $K\in \Perv(\sM)$. 

We set ${}_N\sA_{0,n} := {}_N\sM_{0,n}$. Then, we have the induced map $\tau: \sM \rightarrow \sA$ where $\sA = \prod {}_N\sA_{g_i,n_i}$ and $\tau$ given by the product of  the identities  on ${}_N\sM_{0,n}$ and extended
Torelli maps on other factors. Note that $\sA$ has $PEC$ by Theorem \ref{thm:NAgn}.
Since  $\sY=\tau(\sM)\subset \sA$ is closed by Lemma \ref{lem:imageclosed}, it also has PEC.
Let $\sZ \subset \sM$ be the closed subset defined as the union of the closed subsets of $\sM$ given by replacing the $i$-th factor in $\sM$ by a divisor in the boundary $\sM_{g_i, n_i}^{sc}-\sM_{g_i,n_i}$.
Then $\sZ$ empty if and only if all the factors are either $\sM_{0,n}$ or $\sM_{1,1}$. In this case, we have already observed the conclusion of the theorem in the previous proposition. Therefore we can assume that $\sZ$ is a nonempty  divisor. By an application of Theorem \ref{thm:key1}, we are reduced to showing the $PEC$ property on the boundary $\sZ$.

The boundary $\sZ$ is  a union of divisors $\sD$, each of which is of the above form   (a product of spaces ${}_N\sM_{g',n'}^{sc}$). For each component $\sD$, we claim that
$(H(\sD),S(\sD)) < (H(\sM), S(\sM))$. To see this, note that a boundary divisor $\sD$ amounts to replacing one of the factors in $\sM$ by one of boundary divisors of the given factor. In particular, $S$ for this factor is one less (it is the dimension of this divisor). On the other hand, $H$ is the maximum dimension of the factors, and this either decreases or remains the same. This proves the claim.

It remains to show that $\chi(\sZ, K)\ge 0$ for $K\in \Perv(\sZ)$.
On the other hand, since $\Perv(\sZ)$ is Artinian, we can assume that $K$ is simple.
Then $K$ is the image of a perverse sheaf $L$ in some $\sZ' = \prod {}_N\sM_{g_i', n_i'}$, with $(H(\sZ),S(\sZ)) < (H(\sM),S(\sM))$. Therefore, the claim follows by the induction hypothesis.

\end{proof}

As an immediate consequence we have the following corollary, which completes the proof of Theorem \ref{thm:mainthm}.

\begin{cor}\label{cor:Mg}
 Let $\sM= \sM_{g_1, n_1}\times\ldots \times \sM_{g_r, n_r}$, where either $g_i > 1$ or $g_i=1$ and $n_i>0$. 
   If   $P$ is perverse sheaf on $\sM$, then
    $$\chi^{orb}(\sM,P)\ge 0$$

\end{cor}

\begin{proof}
 Let $\sM^{sc}= \sM_{g_1, n_1}^{sc}\times\ldots \times \sM_{g_r, n_r}^{sc}$. The embedding $j:\sM\to \sM^{sc}$ is affine, since the complement is a divisor.
 Therefore $j_*P\in \Perv(\sM^{sc})$. Consequently,
 $$\chi^{orb}(\sM,P) = \chi^{orb}(\sM^{sc}, j_*P) \ge 0$$
\end{proof}

\section{Remarks in characteristic $p$.}
In this section, we discuss some peculiarities in the characteristic  $p$ setting. In particular, we show that the analog of Theorem \ref{thm:NAgn} is false in this setting. We fix $N \geq 3$, and consider the moduli space $\prescript{}{N}\sA_{g}$, viewed as a scheme over $\operatorname{Spec} \bbZ[1/N]$. Recall, that this is a smooth scheme over the base. Let $\prescript{}{N}\sA_{g} \otimes \bbF_p$ (resp. $\prescript{}{N}\sA_{g} \otimes \bbC$) denote the corresponding moduli space over $\bbF_p$ (resp. $\bbC$). For a fixed prime $p$, let $\chi(\prescript{}{N}\sA_{g} \otimes \bbF_p)$ denote the usual $\ell$-adic Euler characteristic (over $\bar{\bbF}_p$) for some fixed prime $\ell \neq p$. Note that this is an integer independent of $\ell$. Similarly, let $\chi(\prescript{}{N}\sA_{g} \otimes \bbC)$ denote the topological Euler characteristic (say with rational coefficients). Note that this can be identified with the $\ell$-adic Euler characteristic by the usual comparison theorem for \'etale cohomology and singular cohomology. We begin by first observing the following:

\begin{thm}
For all $p \nmid N$, $\chi(\prescript{}{N}\sA_{g} \otimes \bbF_p) = \chi(\prescript{}{N}\sA_{g} \otimes \bbC).$

\end{thm}
\begin{proof}
Choose a good compactification (over $\bbZ[1/N]$) (for example, a toroidal compactification as in \cite{fc}) and use the lemma below.
\end{proof}

The following Lemma is standard, but we give a proof for lack of a reference.

\begin{lem}
Let $S$ be a strictly henselian ring with with closed point $s$ and generic point $\eta$. Let $\bar{X} \rightarrow S$ a smooth proper morphism, $D \subset \bar{X}$ a relative normal crossings divisor, and $X: = \bar{X} \setminus D$. Then 
$\chi(X_{s}) = \chi(X_{\bar{\eta}})$.
\end{lem}
\begin{proof}
Note that the analogous assertion for $\bar{X}$ holds as a corollary to the classical smooth and proper base change theorem. Suppose $D = D_1 \cup \cdots \cup D_n$ is a decomposition into irreducible component, and let $D_I : = \bigcap_{i \in I} D_i$ for $I \subset \{1,\ldots,n\}$. Then each $D_I \rightarrow S$ is smooth proper, and therefore one also has
$\chi((D_I)_s) = \chi((D_I)_{\bar{\eta}})$. By desecent, we conclude that $\chi(D_s) = \chi(D_{\bar{\eta}})$. Alternatively, we can also conclude via the inclusion-exclusion principle. Finally, by additivity of Euler characteristics we get the desired conclusion for $X$.
\end{proof}

\begin{rem}
Consider a smooth compactification of $\prescript{}{N}\sA_{g}$ with normal crossings boundary $\sD$ inducing a smooth compactification of $\prescript{}{N}\sA_{g} \otimes \bbF_p$ with boundary $D$. For any smooth subvariety of $\prescript{}{N}\sA_{g} \otimes \bbF_p$ whose closure has normal crossings with $D$ (so that in particular the closure is smooth) and lifts to characteristic 0, we obtain the signed Euler characteristic property as a consequence of the previous Lemma.
\end{rem}

We now recall some counter-examples in characteristic  p resulting from super-singular strata. For the rest of this section, we assume that $p$ does not divide $N$. Below, we set $\prescript{}{N}\sA_{g,p} := \prescript{}{N}\sA_{g} \otimes \bbF_p[\zeta_N]$ (where $\zeta_N$ is a fixed primitive $N$-th root of unity). Below, if $N =1$, we will sometimes denote $\prescript{}{N}\sA_{g,p}$ simply by $\sA_{g,p}$. We denote by $|\prescript{}{N}\sA_{g,p}|$ the corresponding coarse modulii spaces. If $N > 2$, then this is the same as the fine moduli space. \\

Let $\prescript{}{N}\cS_{g,p} \subset |\prescript{}{N}\sA_{g,p}|$ denote the super-singular locus. If $N=1$, we drop the superscript and denote it by $\cS_{g,p}$. By definition this is the locus of points $(A,\lambda, H)$ (here $\lambda$ is a principal polarization and $H$ is a level structure) such that $A \otimes \bar{\bbF}_p$ is super-singular. This is a closed subset with $\dim(\prescript{}{N}\cS_{g,p}) = \lfloor \frac{g^2}{4} \rfloor$ (\cite{lioort} ). Note that it is known that in general $|\prescript{}{N}\cS_{g,p}|$ is not irreducible, but it is equidimensional. In particular, $\dim(\prescript{}{N}\cS_{2,p}) = 1$  and $\dim(\prescript{}{N}\cS_{3,p}) = 2$. The components of $\prescript{}{N}\cS_{g,p}$ can be parametrized by spaces of flag type data for Dieudonne modules (\cite{lioort}). We do not recall the details, but give a description of the resulting parameter space in the case $g=2$ (\cite{KO} (section 2, 5.1, 5.3) and $g=3$ \cite{lioort} 

\begin{thm}(\cite{KO},\cite{lioort})
\begin{enumerate}
\item The irreducible components of $|\prescript{}{N}\cS_{2,p}|$ are isomorphic to $\bbP^1$.
\item Let $\cX \subset \bbP^2_{\bbF_p}$ denote the Fermat curve of degree $p+1$ (i.e. defined by $x^{p+1}+y^{p+1}+z^{p+1} =0$), and $\cY := \bbP(\cO_{\cX}(1) \oplus \cO_{\cX}(-1))$. There is a natural morphism 
$$ \pi: \cY \rightarrow \cS_{3,p}$$
such that:
\begin{enumerate}
\item $\pi$ is of degree one onto its image.
\item The image of $\pi$ is an irreducible component of $\cS_{3,p}$.
\item Let $\cD \subset \cY$ denote the  the unique section with negative self intersection. Then the morphism $\pi$ contracts $\cD$ to a point. 
\end{enumerate}
\end{enumerate}
\end{thm}

\begin{thm}\label{thm:eulercharsameinp}
\-
\begin{enumerate}
\item There is a smooth projective curve $C \subset \prescript{}{N}\sA_{2,p}$ which does not satisfy the signed Euler characteristic property i.e. the $\ell$-adic Euler characteristic $\chi(C_{\bar{\bbF}_p}) > 0 $. In fact, $C = \bbP^1$.
\item Suppose $p>2$. There exists a normal surface in  $\sA_{3,p}$ with isolated singularities which does not satisfy the signed Euler characteristic property.
\end{enumerate}
\end{thm}
\begin{proof}
The first part is a direct consequence of the first part of the previous theorem. For the second part, we note that the blow down of the surface $\cY$ (i.e. $\pi(\cY)$) in (2) of Theorem \ref{thm:eulercharsameinp} satisfies the properties stated here. Namely, it is a normal surface with isolated singular point, and it has negative Euler characteristic. This follows from the fact that the Fermat curve in this case has negative Euler characteristic, and therefore so does the projective bundle described above. 
\end{proof}

\begin{cor}\label{cor:counterexamplecharp}
\-
\begin{enumerate}
\item There is an embedding $(\bbP^1)^{\lfloor g/2 \rfloor} \subset \prescript{}{N}\sA_{g,p}$ for all $g \geq 2$. 
\item There exists a sequence of  smooth subvarieties $V_{g_i}\subset {}_N A_{g_i,p}$,  such that $(-1)^{\dim V_{g_i}} \chi(V_{g_i})\to -\infty$.
\item The contangent bundle $\Omega^1{ \prescript{}{N}\sA_{g,p}}$ and the Hodge bundle $F^1$ associated to the universal family of 
abelian varieties on $\prescript{}{N}\sA_{g,p}$ are not nef (in the sense that their restrictions to suitable projective subvarieties are not nef).
\end{enumerate}

\end{cor}
\begin{proof}
If $g=2$, (1)  is the first part of the previous theorem. Suppose $g > 2$, and let $d= \lfloor g/2 \rfloor$. If $g$ is odd,  let $E$ be a fixed polarized elliptic curve with level $N$ structure. Then an irreducible component of either
$$\{A_1 \times \ldots \times A_d | A_i \in {}_N\cS_{2,p}\} \quad \text{if $g$ even}$$
$$ \{A_1 \times \ldots \times A_d\times E |A_i \in {}_N\cS_{2,p}\} \quad \text{if $g$ even}$$
gives an embedded copy of $(\bbP^1)^d$ in ${}_N\sA_{g,p}$.

(2) follows immediately from (1).

For the last item, observe that the restriction of $\Omega^1_{ {}_{N}\sA_{g,p}}$ to a smooth rational projective curve is not nef, because 
it has $\mathcal{O}(-2)$  as a quotient. One has that
$$\Omega^1_{ {}_{N}\sA_{g,p}} \cong S^2 F^1$$
e.g. by \cite[chap IV, theorem 7.7]{fc}. Therefore $F^1$ cannot be nef.
    
\end{proof}

\begin{rem}\label{rem:counterexamplesemipositive}
\-
\begin{enumerate}
 \item The last corollary shows that  Theorem \ref{thm:mainthm} fails very badly in positive characteristic.

\item  The corollary gives a counterexample to Theorem  \ref{thm:semipos} on the semipositivity (i.e. nefness) of Hodge bundles in characteristic p. This goes back to Moret-Bailly \cite{mb} who used essentially the same
example when $g=2$.

\end{enumerate}

\end{rem}

\bibliographystyle{plain}
\bibliography{biblio}

@article {alexeev1,
    AUTHOR = {Alexeev, Valery},
     TITLE = {Complete moduli in the presence of semiabelian group action},
   JOURNAL = {Ann. of Math. (2)},
  FJOURNAL = {Annals of Mathematics. Second Series},
    VOLUME = {155},
      YEAR = {2002},
    NUMBER = {3},
     PAGES = {611--708},
      ISSN = {0003-486X,1939-8980},
   MRCLASS = {14K10 (14D20 14M25)},
  MRNUMBER = {1923963},
MRREVIEWER = {J\'anos\ Koll\'ar},
       DOI = {10.2307/3062130},
       URL = {https://doi.org/10.2307/3062130},
}

@article {alexeev2,
    AUTHOR = {Alexeev, Valery},
     TITLE = {Compactified {J}acobians and {T}orelli map},
   JOURNAL = {Publ. Res. Inst. Math. Sci.},
  FJOURNAL = {Kyoto University. Research Institute for Mathematical
              Sciences. Publications},
    VOLUME = {40},
      YEAR = {2004},
    NUMBER = {4},
     PAGES = {1241--1265},
      ISSN = {0034-5318,1663-4926},
   MRCLASS = {14D22 (14C34)},
  MRNUMBER = {2105707},
MRREVIEWER = {Lucia\ Caporaso},
       URL = {http://projecteuclid.org/euclid.prims/1145475446},
}

@incollection {bbd,
    AUTHOR = {Be\u ilinson, A. A. and Bernstein, J. and Deligne, P.},
     TITLE = {Faisceaux pervers},
 BOOKTITLE = {Analysis and topology on singular spaces, {I} ({L}uminy,
              1981)},
    SERIES = {Ast\'erisque},
    VOLUME = {100},
     PAGES = {5--171},
 PUBLISHER = {Soc. Math. France, Paris},
      YEAR = {1982},
   MRCLASS = {32C38},
  MRNUMBER = {751966},
MRREVIEWER = {Zoghman\ Mebkhout},
}

@incollection {beilinson,
    AUTHOR = {Be\u ilinson, A. A.},
     TITLE = {How to glue perverse sheaves},
 BOOKTITLE = {{$K$}-theory, arithmetic and geometry ({M}oscow, 1984--1986)},
    SERIES = {Lecture Notes in Math.},
    VOLUME = {1289},
     PAGES = {42--51},
 PUBLISHER = {Springer, Berlin},
      YEAR = {1987},
      ISBN = {3-540-18571-2},
   MRCLASS = {14F99 (18E25 32C38)},
  MRNUMBER = {923134},
MRREVIEWER = {Jean-Luc\ Brylinski},
       DOI = {10.1007/BFb0078366},
       URL = {https://doi.org/10.1007/BFb0078366},
}

@article {Beh,
    AUTHOR = {Behrend, Kai},
     TITLE = {Donaldson-{T}homas type invariants via microlocal geometry},
   JOURNAL = {Ann. of Math. (2)},
  FJOURNAL = {Annals of Mathematics. Second Series},
    VOLUME = {170},
      YEAR = {2009},
    NUMBER = {3},
     PAGES = {1307--1338},
      ISSN = {0003-486X,1939-8980},
   MRCLASS = {14N35 (14C15 14C17 14D23)},
  MRNUMBER = {2600874},
MRREVIEWER = {Hsian-Hua\ Tseng},
       DOI = {10.4007/annals.2009.170.1307},
       URL = {https://doi.org/10.4007/annals.2009.170.1307},
}

@book {blr,
    AUTHOR = {Bosch, Siegfried and L\"utkebohmert, Werner and Raynaud,
              Michel},
     TITLE = {N\'eron models},
    SERIES = {Ergebnisse der Mathematik und ihrer Grenzgebiete (3) [Results
              in Mathematics and Related Areas (3)]},
    VOLUME = {21},
 PUBLISHER = {Springer-Verlag, Berlin},
      YEAR = {1990},
     PAGES = {x+325},
      ISBN = {3-540-50587-3},
   MRCLASS = {14K15 (11G10 14L15)},
  MRNUMBER = {1045822},
MRREVIEWER = {James\ Milne},
       DOI = {10.1007/978-3-642-51438-8},
       URL = {https://doi.org/10.1007/978-3-642-51438-8},
}

@article {dm,
    AUTHOR = {Deligne, P. and Mumford, D.},
     TITLE = {The irreducibility of the space of curves of given genus},
   JOURNAL = {Inst. Hautes \'Etudes Sci. Publ. Math.},
  FJOURNAL = {Institut des Hautes \'Etudes Scientifiques. Publications
              Math\'ematiques},
    NUMBER = {36},
      YEAR = {1969},
     PAGES = {75--109},
      ISSN = {0073-8301,1618-1913},
   MRCLASS = {14.20},
  MRNUMBER = {262240},
MRREVIEWER = {Manfred\ Herrmann},
       URL = {http://www.numdam.org/item?id=PMIHES_1969__36__75_0},
}

@misc{dw,
Author = {Ya Deng and Botong Wang},
Title = {Linear Chern-Hopf-Thurston conjecture},
Year = {2024},
Eprint = {arXiv:2405.12012},
}

@book {fc,
    AUTHOR = {Faltings, Gerd and Chai, Ching-Li},
     TITLE = {Degeneration of abelian varieties},
    SERIES = {Ergebnisse der Mathematik und ihrer Grenzgebiete (3) [Results
              in Mathematics and Related Areas (3)]},
    VOLUME = {22},
      NOTE = {With an appendix by David Mumford},
 PUBLISHER = {Springer-Verlag, Berlin},
      YEAR = {1990},
     PAGES = {xii+316},
      ISBN = {3-540-52015-5},
   MRCLASS = {14K10 (11G10 14D20 14K25)},
  MRNUMBER = {1083353},
MRREVIEWER = {Min\ Ho\ Lee},
       DOI = {10.1007/978-3-662-02632-8},
       URL = {https://doi.org/10.1007/978-3-662-02632-8},
}

@article {ffs,
    AUTHOR = {Fujino, Osamu and Fujisawa, Taro and Saito, Morihiko},
     TITLE = {Some remarks on the semipositivity theorems},
   JOURNAL = {Publ. Res. Inst. Math. Sci.},
  FJOURNAL = {Publications of the Research Institute for Mathematical
              Sciences},
    VOLUME = {50},
      YEAR = {2014},
    NUMBER = {1},
     PAGES = {85--112},
      ISSN = {0034-5318,1663-4926},
   MRCLASS = {14D07 (32G20)},
  MRNUMBER = {3167580},
MRREVIEWER = {Christian\ Schnell},
       DOI = {10.4171/PRIMS/125},
       URL = {https://doi.org/10.4171/PRIMS/125},
}

@article {fk,
    AUTHOR = {Franecki, J. and Kapranov, M.},
     TITLE = {The {G}auss map and a noncompact {R}iemann-{R}och formula for
              constructible sheaves on semiabelian varieties},
   JOURNAL = {Duke Math. J.},
  FJOURNAL = {Duke Mathematical Journal},
    VOLUME = {104},
      YEAR = {2000},
    NUMBER = {1},
     PAGES = {171--180},
      ISSN = {0012-7094,1547-7398},
   MRCLASS = {14C40 (32C38 32S60)},
  MRNUMBER = {1769729},
MRREVIEWER = {Laurent\ Manivel},
       DOI = {10.1215/S0012-7094-00-10417-6},
       URL = {https://doi.org/10.1215/S0012-7094-00-10417-6},
}

@article {reich,
    AUTHOR = {Reich, Ryan},
     TITLE = {Notes on {B}eilinson's ``{H}ow to glue perverse sheaves''
              [MR0923134]},
   JOURNAL = {J. Singul.},
  FJOURNAL = {Journal of Singularities},
    VOLUME = {1},
      YEAR = {2010},
     PAGES = {94--115},
      ISSN = {1949-2006},
   MRCLASS = {14F05 (18E30)},
  MRNUMBER = {2671769},
MRREVIEWER = {Jon\ Eivind\ Vatne},
}

@article {gromov,
    AUTHOR = {Gromov, M.},
     TITLE = {K\"ahler hyperbolicity and {$L_2$}-{H}odge theory},
   JOURNAL = {J. Differential Geom.},
  FJOURNAL = {Journal of Differential Geometry},
    VOLUME = {33},
      YEAR = {1991},
    NUMBER = {1},
     PAGES = {263--292},
      ISSN = {0022-040X,1945-743X},
   MRCLASS = {58G10 (32C17 58A14)},
  MRNUMBER = {1085144},
MRREVIEWER = {J\'ozef\ Dodziuk},
       URL = {http://projecteuclid.org/euclid.jdg/1214446039},
}

@article {harder,
    AUTHOR = {Harder, G.},
     TITLE = {A {G}auss-{B}onnet formula for discrete arithmetically defined
              groups},
   JOURNAL = {Ann. Sci. \'Ecole Norm. Sup. (4)},
  FJOURNAL = {Annales Scientifiques de l'\'Ecole Normale Sup\'erieure.
              Quatri\`eme S\'erie},
    VOLUME = {4},
      YEAR = {1971},
     PAGES = {409--455},
      ISSN = {0012-9593},
   MRCLASS = {20H10 (10D25 22E40 53C30 57F15)},
  MRNUMBER = {309145},
MRREVIEWER = {H.\ Garland},
       URL = {http://www.numdam.org/item?id=ASENS_1971_4_4_3_409_0},
}

@article {hz,
    AUTHOR = {Harer, J. and Zagier, D.},
     TITLE = {The {E}uler characteristic of the moduli space of curves},
   JOURNAL = {Invent. Math.},
  FJOURNAL = {Inventiones Mathematicae},
    VOLUME = {85},
      YEAR = {1986},
    NUMBER = {3},
     PAGES = {457--485},
      ISSN = {0020-9910,1432-1297},
   MRCLASS = {32G15 (14H15 57R20)},
  MRNUMBER = {848681},
MRREVIEWER = {William\ Abikoff},
       DOI = {10.1007/BF01390325},
       URL = {https://doi.org/10.1007/BF01390325},
}

@incollection {kato,
    AUTHOR = {Kato, Kazuya},
     TITLE = {Logarithmic structures of {F}ontaine-{I}llusie},
 BOOKTITLE = {Algebraic analysis, geometry, and number theory ({B}altimore,
              {MD}, 1988)},
     PAGES = {191--224},
 PUBLISHER = {Johns Hopkins Univ. Press, Baltimore, MD},
      YEAR = {1989},
      ISBN = {0-8018-3841-X},
   MRCLASS = {14F30 (14G20)},
  MRNUMBER = {1463703},
MRREVIEWER = {Adolfo\ Quir\'os},
}

@article {KO,
    AUTHOR = {Katsura, Toshiyuki and Oort, Frans},
     TITLE = {Families of supersingular abelian surfaces},
   JOURNAL = {Compositio Math.},
  FJOURNAL = {Compositio Mathematica},
    VOLUME = {62},
      YEAR = {1987},
    NUMBER = {2},
     PAGES = {107--167},
      ISSN = {0010-437X,1570-5846},
   MRCLASS = {14K10 (14D22 14K15)},
  MRNUMBER = {898731},
MRREVIEWER = {Joseph\ H.\ Silverman},
       URL = {http://www.numdam.org/item?id=CM_1987__62_2_107_0},
}

@book {kw,
    AUTHOR = {Kiehl, Reinhardt and Weissauer, Rainer},
     TITLE = {Weil conjectures, perverse sheaves and {$l$}'adic {F}ourier
              transform},
    SERIES = {Ergebnisse der Mathematik und ihrer Grenzgebiete. 3. Folge. A
              Series of Modern Surveys in Mathematics [Results in
              Mathematics and Related Areas. 3rd Series. A Series of Modern
              Surveys in Mathematics]},
    VOLUME = {42},
 PUBLISHER = {Springer-Verlag, Berlin},
      YEAR = {2001},
     PAGES = {xii+375},
      ISBN = {3-540-41457-6},
   MRCLASS = {14F20 (11G25 18E30 20G05)},
  MRNUMBER = {1855066},
MRREVIEWER = {James\ Milne},
       DOI = {10.1007/978-3-662-04576-3},
       URL = {https://doi.org/10.1007/978-3-662-04576-3},
}

@article {knudsen,
    AUTHOR = {Knudsen, Finn F.},
     TITLE = {The projectivity of the moduli space of stable curves. {II}.
              {T}he stacks {$M\sb{g,n}$}},
   JOURNAL = {Math. Scand.},
  FJOURNAL = {Mathematica Scandinavica},
    VOLUME = {52},
      YEAR = {1983},
    NUMBER = {2},
     PAGES = {161--199},
      ISSN = {0025-5521,1903-1807},
   MRCLASS = {14H10 (14D20 14D22)},
  MRNUMBER = {702953},
MRREVIEWER = {P.\ E.\ Newstead},
       DOI = {10.7146/math.scand.a-12001},
       URL = {https://doi.org/10.7146/math.scand.a-12001},
}

@article {lo,
    AUTHOR = {Laszlo, Yves and Olsson, Martin},
     TITLE = {Perverse {$t$}-structure on {A}rtin stacks},
   JOURNAL = {Math. Z.},
  FJOURNAL = {Mathematische Zeitschrift},
    VOLUME = {261},
      YEAR = {2009},
    NUMBER = {4},
     PAGES = {737--748},
      ISSN = {0025-5874,1432-1823},
   MRCLASS = {14A20 (14D20)},
  MRNUMBER = {2480756},
MRREVIEWER = {Gerhard\ Pfister},
       DOI = {10.1007/s00209-008-0348-z},
       URL = {https://doi.org/10.1007/s00209-008-0348-z},
}

@book {laumon,
    AUTHOR = {Laumon, G\'erard and Moret-Bailly, Laurent},
     TITLE = {Champs alg\'ebriques},
    SERIES = {Ergebnisse der Mathematik und ihrer Grenzgebiete. 3. Folge. A
              Series of Modern Surveys in Mathematics [Results in
              Mathematics and Related Areas. 3rd Series. A Series of Modern
              Surveys in Mathematics]},
    VOLUME = {39},
 PUBLISHER = {Springer-Verlag, Berlin},
      YEAR = {2000},
     PAGES = {xii+208},
      ISBN = {3-540-65761-4},
   MRCLASS = {14A20 (14D20)},
  MRNUMBER = {1771927},
MRREVIEWER = {Dan\ Edidin},
}

@article{laumonEuler,
 author = {Laumon, Gerard},
 title = {Comparaison de caract{\'e}ristiques d'{Euler}-{Poincar{\'e}} en cohomologie l-adique},
 fjournal = {Comptes Rendus de l'Acad{\'e}mie des Sciences. S{\'e}rie I},
 journal = {C. R. Acad. Sci., Paris, S{\'e}r. I},
 issn = {0764-4442},
 volume = {292},
 pages = {209--212},
 year = {1981},
 language = {French},
 zbMATH = {3734080},
 Zbl = {0468.14005}
}

@book {lazarsfeld,
    AUTHOR = {Lazarsfeld, Robert},
     TITLE = {Positivity in algebraic geometry. {I}},
    SERIES = {Ergebnisse der Mathematik und ihrer Grenzgebiete. 3. Folge. A
              Series of Modern Surveys in Mathematics [Results in
              Mathematics and Related Areas. 3rd Series. A Series of Modern
              Surveys in Mathematics]},
    VOLUME = {48},
      NOTE = {Classical setting: line bundles and linear series},
 PUBLISHER = {Springer-Verlag, Berlin},
      YEAR = {2004},
     PAGES = {xviii+387},
      ISBN = {3-540-22533-1},
   MRCLASS = {14-02 (14C20)},
  MRNUMBER = {2095471},
MRREVIEWER = {Mihnea\ Popa},
       DOI = {10.1007/978-3-642-18808-4},
       URL = {https://doi.org/10.1007/978-3-642-18808-4},
}

@book {lioort,
    AUTHOR = {Li, Ke-Zheng and Oort, Frans},
     TITLE = {Moduli of supersingular abelian varieties},
    SERIES = {Lecture Notes in Mathematics},
    VOLUME = {1680},
 PUBLISHER = {Springer-Verlag, Berlin},
      YEAR = {1998},
     PAGES = {iv+116},
      ISBN = {3-540-63923-3},
   MRCLASS = {14K10 (11G10 14L05)},
  MRNUMBER = {1611305},
MRREVIEWER = {Ben\ Moonen},
       DOI = {10.1007/BFb0095931},
       URL = {https://doi.org/10.1007/BFb0095931},
}

@article {LMW,
    AUTHOR = {Liu, Yongqiang and Maxim, Lauren\c tiu and Wang, Botong},
     TITLE = {Aspherical manifolds, {M}ellin transformation and a question
              of {B}obadilla-{K}oll\'ar},
   JOURNAL = {J. Reine Angew. Math.},
  FJOURNAL = {Journal f\"ur die Reine und Angewandte Mathematik. [Crelle's
              Journal]},
    VOLUME = {781},
      YEAR = {2021},
     PAGES = {1--18},
      ISSN = {0075-4102,1435-5345},
   MRCLASS = {32Q55 (14A30 14F45)},
  MRNUMBER = {4343101},
MRREVIEWER = {Clara\ L\"oh},
       DOI = {10.1515/crelle-2021-0055},
       URL = {https://doi.org/10.1515/crelle-2021-0055},
}

@incollection {mb,
    AUTHOR = {Moret-Bailly, Laurent},
     TITLE = {Familles de courbes et de vari\'et\'es ab\'eliennes sur
              {${\Bbb P}^1$}. {II}. {E}xemples},
      NOTE = {Seminar on Pencils of Curves of Genus at Least Two},
   JOURNAL = {Ast\'erisque},
  FJOURNAL = {Ast\'erisque},
    NUMBER = {86},
      YEAR = {1981},
     PAGES = {125--140},
      ISSN = {0303-1179,2492-5926},
   MRCLASS = {14K05 (14K25)},
  MRNUMBER = {3618576},
}

@book {namikawa,
    AUTHOR = {Namikawa, Yukihiko},
     TITLE = {Toroidal compactification of {S}iegel spaces},
    SERIES = {Lecture Notes in Mathematics},
    VOLUME = {812},
 PUBLISHER = {Springer, Berlin},
      YEAR = {1980},
     PAGES = {viii+162},
      ISBN = {3-540-10021-0},
   MRCLASS = {32J05 (14K10 32M99 32N15)},
  MRNUMBER = {584625},
MRREVIEWER = {Masa-Nori\ Ishida},
}

@incollection {Verdier,
    AUTHOR = {Verdier, J.-L.},
     TITLE = {Sp\'ecialisation de faisceaux et monodromie mod\'er\'ee},
 BOOKTITLE = {Analysis and topology on singular spaces, {II}, {III}
              ({L}uminy, 1981)},
    SERIES = {Ast\'erisque},
    VOLUME = {101-102},
     PAGES = {332--364},
 PUBLISHER = {Soc. Math. France, Paris},
      YEAR = {1983},
   MRCLASS = {32C38 (14D05 32C99)},
  MRNUMBER = {737938},
MRREVIEWER = {Jean-Paul\ Brasselet},
}

@article {wz,
    AUTHOR = {Wu, Lei and Zhou, Peng},
     TITLE = {Log {$\mathcal{D}$}-modules and index theorems},
   JOURNAL = {Forum Math. Sigma},
  FJOURNAL = {Forum of Mathematics. Sigma},
    VOLUME = {9},
      YEAR = {2021},
     PAGES = {Paper No. e3, 32},
      ISSN = {2050-5094},
   MRCLASS = {14F10 (14A21 32C38)},
  MRNUMBER = {4202488},
MRREVIEWER = {Corrado\ Marastoni},
       DOI = {10.1017/fms.2020.62},
       URL = {https://doi.org/10.1017/fms.2020.62},
}

@article {AW,
    AUTHOR = {Arapura, Donu and Wang, Botong},
     TITLE = {Perverse sheaves on varieties with large fundamental groups},
   JOURNAL = {J. Differential Geom.},
  FJOURNAL = {Journal of Differential Geometry},
    VOLUME = {129},
      YEAR = {2025},
    NUMBER = {1},
     PAGES = {1--15},
      ISSN = {0022-040X,1945-743X},
   MRCLASS = {32 (14)},
  MRNUMBER = {4856131},
       DOI = {10.4310/jdg/1736261441},
       URL = {https://doi.org/10.4310/jdg/1736261441},
}

@misc{morel,
    title={Beilinson's constructin of nearby cycles and gluing},
    author={Morel, S},
    url={http://perso.ens-lyon.fr/sophie.morel/}
}

@misc{stacks-project,
  author       = {The {Stacks project authors}},
  title        = {The Stacks project},
  howpublished = {\url{https://stacks.math.columbia.edu}},
  year         = {2025},
}

\end{document}